\documentclass[11pt, onepage,reqno]{amsart}
\usepackage{a4wide}
\usepackage[usenames,dvipsnames]{color}
\usepackage{amsmath}%
\usepackage{amsfonts}%
\usepackage{amssymb}%
\usepackage{amsthm}%
\usepackage{graphicx}
\usepackage{pgf,tikz}
\usepackage{mathrsfs}
\usetikzlibrary{arrows}
\usepackage{tikz-cd, floatflt, pgfplots}
\usepackage[all,cmtip]{xy}
\usepackage{multirow}
\usepackage{mathrsfs}
\usepackage{array}
\usepackage{rotating}
\usepackage{multirow}
\usepackage{enumitem}
\usepackage{pstricks-add}
\usepackage{pst-func}
\usepackage{mathtools}
\usepackage{extpfeil}
\usepackage[new]{old-arrows}
\usepackage{csquotes}
\usepackage{epigraph}
\usepackage{bm}

\usepackage[utf8]{inputenc}
\usepackage[english]{babel}

\usepackage{hyperref}
\hypersetup{
    colorlinks,
    citecolor = black,
    filecolor=black,
    linkcolor=black,
    urlcolor=black
}

\newtheorem{theorem}{Theorem}[section]
\newtheorem*{theorem*}{Theorem}

\newtheorem{corollary}[theorem]{Corollary}

\newtheorem{lemma}[theorem]{Lemma}
\newtheorem*{lemma*}{Lemma}

\newtheorem{proposition}[theorem]{Proposition}

\theoremstyle{definition}
\newtheorem{definition}[theorem]{Definition} 
\newtheorem{example}[theorem]{Example}

\theoremstyle{remark}
\newtheorem{remark}[theorem]{Remark}
\makeatletter
\def\@marginparreset{%
  \reset@font
  \tiny
  \@marginparfont
  \@setminipage
}
\def\@marginparfont{\tiny} 
\makeatother

\newcommand{\todo[1]}{\textcolor{red!65!black}{(*)}\marginpar{\textbf{DO:} \textcolor{red!65!black}{#1}}}

\def\C{\mathbb{C}}
\def\R{\mathbb{R}}
\def\Q{\mathbb{Q}}
\def\Z{\mathbb{Z}}

\def\Proj{\mathbb{P}}

\def\Hh{\mathcal{H}}

\def\Mm{\mathcal{M}}

\def\Pp{\mathcal{P}}

\usepackage{mathrsfs}

\def\cD{\mathscr{D}}

\def\cF{\mathscr{F}}

\def\map{\longrightarrow}

\def\norm{\vartriangleleft}
\def\rk{\mathrm{rk}}

\def\inj{\hookrightarrow}

\def\surj{\twoheadrightarrow}
\def\longinj{\longhookrightarrow}

\newcommand{\HomLin}[3]{\mathrm{Hom}_{#1}\left(#2,#3\right)}

\def\id{\mathrm{id}}
\def\Im{\mathrm{Im}}
\def\ker{\mathrm{ker}}

\DeclareMathOperator{\lcm}{lcm}

\DeclareMathOperator{\disc}{disc}

\def\tq{\mid}

\newcommand{\ord}[2]{\mathrm{ord}_{#2}\left(#1\right)}

\def\Aut{\mathrm{Aut}}

\def\div{\mathrm{div}}

\def\cover{\rho}

\title[Groups acting on moduli spaces of hyper-Kähler manifolds]{Groups acting on moduli spaces of hyper-Kähler manifolds}
\author{Francesca Rizzo}
\address{Université Paris Cité, IMJ - PRG, 75013 Paris, France}
\email{francesca.rizzo@imj-prg.fr}

\begin{document}
\maketitle
\begin{abstract}
\vspace*{-3em}
   The period morphism of polarized hyper-Kähler manifolds of K3$^{[m]}$-type gives an embedding of each connected component of the moduli space of polarized hyper-Kähler manifolds of K3$^{[m]}$-type into their period space, which is the quotient of a Hermitian symmetric domain by an arithmetic group. Following work of Stellari and Gritsenko-Hulek-Sankaran, we study the ramification of covering maps between these period spaces that arise from the action of some groups of isometries.
\end{abstract}
\section*{Introduction}
In \cite{stellari2008}, Stellari studied the action of a group of symmetries on the moduli space of polarized K3 surfaces. More precisely, let $h$ be a primitive vector of square $2d$ in the unimodular lattice $\Lambda_{K3}$ associated to the second integral cohomology group of a K3 surface. The period morphism of polarized K3 surfaces of degree $2d$ is an open embedding
$$
\wp: M_{2d} \longinj \cF_{2d}=\cD_{h^\perp}/O(\Lambda_{K3}, h)
$$
of the moduli space $M_{2d}$ of polarized K3 surfaces into the period space $\cF_{2d}$, which is the quotient of the period domain $\cD_{h^\perp}$ associated to the lattice $h^\perp$ (an open analytic subset of a quadric) by the arithmetic group of isometries of the lattice $\Lambda_{K3}$ that fix the vector $h$. The group $O(\Lambda_{K3}, h)$ is a normal subgroup of the group of isometries $O(h^\perp)$ of $h^\perp$, and there is a natural action of the group $G = O(h^\perp)/O(\Lambda_{K3}, h)$ on the period space $\cF_{2d}$. Therefore, the group $G$ (which is an abelian group of exponent $2$) induces a Galois cover
$$
\cover: \cF_{2d} \map \cF_{2d}/G
$$
and acts birationally on the moduli space $M_{2d}$.\\

Stellari characterized the divisorial components of the ramification of the cover $\cover$. The aim of this paper is to generalize Stellari's result to some moduli spaces of polarized hyper-Kähler manifolds.\\

We consider the following more general situation. Fixing an even lattice $\Lambda$ of signature $(2, n_-)$, with $n_-\ge 2$, and a subgroup $\Gamma <O(\Lambda)$ of finite index, we  consider the period space
$\cD_{\Lambda}/\Gamma$, where $\cD_{\Lambda}$ is a Hermitian symmetric domain with a natural action of the group $O(\Lambda)$ of isometries of the lattice $\Lambda$. The period spaces $\cD_\Lambda/\Gamma$ are  normal quasi-projective varieties {\cite[Theorem~6.1.13]{K3Huy}}.\\

When $\Gamma$ is a normal subgroup of some subgroup $O$ of $O(\Lambda)$, we obtain a Galois cover
$$
    \cover: \cD_\Lambda/\Gamma\map \cD_\Lambda/O.
$$
Our aim is to study the divisorial components of the ramification of $\cover$.\\

A \emph{Heegner divisor} is the image in $\cD_\Lambda/\Gamma$ of the hypersurface $\Proj(\beta^\perp)\cap \cD_{\Lambda}$, for some negative square vector $\beta\in \Lambda$. We denote this divisor by $\Hh_{\beta^\perp}$.

We say that a nonzero vector $\beta\in \Lambda$ defines a reflection if there exists an isometry $r_\beta\in O$ which is the identity on $\beta^\perp$ and acts as $-\id$ on $\Z\beta$. If $\beta$ is a vector of negative square that defines a reflection, the Heegner divisor $\Hh_{\beta^\perp}$ is fixed by $[r_\beta] \in O/\Gamma$.\\

Following \cite{GritsenkoRefl2007}, in Theorem~\ref{thm:divisori_ref}, we characterize the divisorial components of the ramification of $\cover$ for even lattices $\Lambda$ of signature $(2, n_-)$, with $n_-\ge 2$, and all normal subgroups $\Gamma\norm O$ of finite index such that $\cD_{\Lambda}/\Gamma$ is irreducible. In particular we prove that the divisorial components of the ramification of $\cover$ are the Heegner divisors $\Hh_{\beta^\perp}$, for primitive vectors $\beta\in \Lambda$ of negative square that define nontrivial classes $[\pm r_\beta]$ in the Galois group $O/\Gamma$.\\

We apply these results to the following geometric situation. A hyper-Kähler manifold of K$3^{[m]}$-type is a smooth deformation of the $m$-th Hilbert scheme of points of a K3 surface.
Given a hyper-Kähler manifold $X$ of K3$^{[m]}$-type, the abelian group $H^2(X,\Z)$ is free of rank 23 and it is equipped with the Beauville–Bogomolov–Fujiki form $q_X$, a non-degenerate $\Z$-valued quadratic form of signature $(3, 20)$. The group $H^2(X,\Z)$ with the quadratic form $q_X$ is an even lattice isomorphic to
\begin{equation}\label{eqn:lattice_K3[n]}
\Lambda_{K3^{[m]}}=\Lambda_{K3} \oplus \Z \ell,
\end{equation}
where $\ell$ is a vector of square $q(\ell)=-2(m-1)$.\\

A polarization on $X$ is the class $H$ of an ample line bundle on $X$ that is primitive in the lattice $H^2(X,\Z) = \Lambda_{K3^{[m]}}$. The polarization type of $(X,H)$ is the $O(\Lambda_{K3^{[m]}})$-orbit of the class $H$.
Smooth polarized hyper-Kähler manifolds $(X, H)$ of K$3^{[m]}$-type of polarization type $\tau = O(\Lambda_{K3^{[m]}})h$ admit a, possibly reducible, quasi-projective coarse moduli space ${^{{[m]}}M_{\tau}}$. The period morphism of polarized hyper-Kähler manifolds of polarization type $\tau$ is the morphism
$$
\wp: {^{{[m]}}M_{\tau}} \map \cD_{h^\perp}/\widehat O(\Lambda_{K3^{[m]}}, h),
$$
where $\widehat O(\Lambda_{K3^{[m]}}, h)$ is the group of isometries of $\Lambda_{K3^{[m]}}$ that fix the vector $h$ and act as $\pm \id$ on the discriminant group of $\Lambda_{K3^{[m]}}$. The period morphism $\wp$ is an open embedding on each connected component of ${^{{[m]}}M_{\tau}}$.\\

In Section~\ref{sec:isometries_lattices_nomality} we study the normality of the subgroup $\widehat O(\Lambda_{K3^{[m]}}, h)$ of $O(h^\perp)$. When the subgroup $\widehat O(\Lambda_{K3^{[m]}},h)$ is normal, the group $G = O(h^\perp)/\widehat O(\Lambda_{K3^{[m]}},h)$ acts on the period space $\cD_{h^\perp}/\widehat O(\Lambda_{K3^{[m]}}, h)$, hence also, birationally, on the moduli space ${^{{[m]}}M_{\tau}}$. We can apply Theorem~\ref{thm:divisori_ref} to characterize the divisorial ramification components of the Galois cover
$$
\cover: \cD_{h^\perp}/\widehat O(\Lambda_{K3^{[m]}}, h) \map \cD_{h^\perp}/O(h^\perp).
$$

They are Heegner divisors associated to vectors $\beta$ such that $[\pm r_\beta]$ are nontrivial elemets of the Galois group $G$. In Theorem~\ref{thm:HK_beta_reflection_numerical} we give, in our situation, a numerical characterization of these vectors when the polarization type has divisibility $1$.\\ 

In dimension 4, the polarization type only depends on the square $h^2 \coloneqq q(h)$ and the divisibility of $h$, the positive generator of the ideal $h\cdot \Lambda_{K3^{[2]}}\subset \Z$. In that case, our result translates into the following theorem.
\begin{theorem} Let $h$ be a polarization of square $2d$ and divisibility $1$. The divisorial components of the ramification of the Galois cover
$$
\cover: \cD_{h^\perp}/\widehat O(\Lambda_{K3^{[2]}}, h) \map \cD_{h^\perp}/O(h^\perp)
$$
are the Heegner divisors $\Hh_{\beta^\perp}$ such that $\beta$ is primitive and satisfies the conditions
\begin{samepage}
\begin{enumerate}[label=\rm{\alph*})]
        \item {$\beta^2\mid 2\div(\beta)$;}
        \item {$\beta^2\neq -2$ and if $\beta^2 = -2d$, then $2d\nmid \beta\cdot \ell$.}
    \end{enumerate}
\end{samepage}
\end{theorem}

\textbf{Case $\boldsymbol{d=1}$}. General hyper-Kähler fourfolds of K3$^{[2]}$-type with a polarization of square $2$ (the divisibility is automatically $1$) are double EPW sextics \cite{OG06-EPWgeneral}. The Galois cover $\cover$ has degree $2$ and the associated involution is the so-called ``duality involution.'' Using the results of \cite{Image_per_map_HK4} on the image of the period map, our theorem shows that $\cover$ has a unique ramification divisor, $\cD_4$, and that $\cD_4$ meets the image of the period map. The divisor induced by $\cD_4$ on the moduli space of hyper-Kähler fourfolds of square $2$ does not meet the open locus of double EPW sextics.

\subsection*{Acknowledgement}
This paper is a reworking of results from my master thesis.  I would like to thank my
advisor Olivier Debarre for proposing me this problem, and for following me closely during the writing phase with many corrections and suggestions.


\section{An introduction to lattice theory}\label{sec:prel_latt}
    A \emph{lattice} $\Lambda$ is a free $\Z$-module of finite rank with a nondegenerate integral symmetric bilinear form $q$. The lattice $\Lambda$ is called \emph{even} if $q$ takes only even values. For each field $K$ containing $\Q$, we denote by $\Lambda_K$ the vector space $\Lambda\otimes_\Z K$. It is endowed with the extension $q_K$ of the bilinear form $q$, which is still nondegenerate.
    The signature of $\Lambda$ is the signature of $q_\R$ and will be denoted by $(n_+, n_ -)$.
    If $n_+$ or $n_-$ is zero, the lattice is called \emph{definite}; otherwise, $\Lambda$ is \emph{indefinite}. The \emph{dual lattice} of $\Lambda$ is
        $$
        \Lambda^\vee \coloneqq \{x\in \Lambda_\Q \mid  \forall y\in \Lambda \ \ x\cdot y\in \Z \} = \HomLin{\Z}{\Lambda}{\Z}.
        $$
    Clearly, there is an inclusion $\Lambda \inj \Lambda^\vee$. The \emph{discriminant group} of $\Lambda$ is the finite abelian group
        $$
        A_\Lambda = \Lambda^\vee / \Lambda.
$$
    We denote by $\disc(\Lambda)$ the cardinality of $A_\Lambda$. The lattice is called \emph{unimodular} if $A_\Lambda$ is trivial or equivalently if $\Lambda ^\vee = \Lambda$. For each $x\in\Lambda$ nonzero, the \emph{divisibility} of $x$, denoted by $\div(x)$, is the positive generator of the ideal $x\cdot \Lambda \subset \Z$.
    Thus, the element
    $
    x_* = \left[\frac{x}{\div(x)}\right]
    $
    is an element of $A_\Lambda$ of order $\div(x)$. The \emph{length} of a lattice $\Lambda$, denoted by $\ell(\Lambda)$, is the minimal number of generators of its discriminant group.\\

    When $\Lambda$ is an even lattice, we obtain a quadratic form $q_\Lambda$ on the discriminant group $A_\Lambda$ with values in $\Q/2\Z$, given by $q_\Lambda([x]) \equiv q_\Q(x) \pmod{2\Z}$ for all $x\in \Lambda^\vee$.
The \emph{group of isometries} of $A_\Lambda$, denoted by $O(A_\Lambda)$, is the group of group automorphisms of $A_\Lambda$ that preserve $q_\Lambda$.\\


    We will denote by $\Z$ the lattice of rank 1 with intersection matrix 1. More generally, we will denote by $\Z(n)$ the lattice of rank 1 with intersection matrix $n$, for all $n>0$, and write $\Z(n)=\Z k$ if the lattice is generated by the vector $k$. Moreover, we let $U$ (the hyperbolic plane) be the even unimodular lattice generated by two vectors $e$ and $f$ such that $e^2 =f^2 = 0$ and $e\cdot f = 1$. There is a unique positive definite  even unimodular lattice of rank 8, which we denote by $E_8$. We indicate by $E_8(-1)$ the lattice obtained by inverting the sign of the quadratic form on $E_8$.\\

    Observe that each isometry $f$ of $\Lambda$ induces an isometry of the discriminant group $A_\Lambda$, given by $r(f)([x]) = [f_\Q(x)]$ for all $[x]\in A_\Lambda$.

    \begin{theorem}[{\cite[Theorem~1.14.2]{Nikulin}}]\label{thm:Nik_surjMor}
        Let $\Lambda$ be an even indefinite lattice with ${\ell(\Lambda) + 2 \le \rk(\Lambda)}$. Then the morphism
        $
        r : O(\Lambda)\to O(A_\Lambda)
        $
        is surjective.
    \end{theorem}
    We denote by $\widetilde O(\Lambda)$ the kernel of this morphism and we call it the \emph{stable orthogonal group}. We also define the group
    $$\widehat O(\Lambda) = \{f \in O(\Lambda)\mid \bar{f} =\pm \id \in O(A_\Lambda)\}.$$ Clearly, $\widetilde O(\Lambda)$ is a subgroup of $\widehat O(\Lambda)$ of index at most 2 and both $\widetilde O(\Lambda)$ and $\widehat O(\Lambda)$ are normal subgroups of $O(\Lambda)$.\\

    We will use the following result, proved in \cite[Satz~10.4]{Eichler}.
    \begin{lemma}[Eichler]\label{lemma:eichler}
        Let $\Lambda$ be a even lattice containing the direct sum of two hyperbolic planes. The $\widetilde O(\Lambda)$-orbit of a primitive vector $h$ is uniquely determined by the integer $h^2$ and the element $h_* = [h/\div(h)]$ of $A_\Lambda$.
    \end{lemma}




    \subsection{Extension of isometries of a sublattice}\label{sec:prelim_iso}
        Let $M$ be a primitive sublattice of an even lattice $L$. We characterize isometries of $M^\perp$ that extend to isometries of $L$, following \cite[Section~1.5]{Nikulin}.\\

        We define
        $$
        O(L, M) = \{f\in O(L) \tq f\vert_M =\id \},
        $$
        the group of isometries of $L$ that are the identity on $M$. Analogously, we define the groups $\widetilde O(L, M) = \widetilde O(L)\cap O(L,M) $ and $\widehat O(L, M) = \widehat O(L)\cap O(L, M)$.

        Clearly, each isometry in $O(L, M)$ restricts to an isometry of $M^\perp$; namely, we have a restriction morphism
        $$
        \mathrm{res}: O(L, M) \map O(M^\perp).
        $$
        We say that an isometry of $O(M^\perp)$ extends to an isometry of $O(L, M)$ if it is in the image of this restriction morphism. \\

        Consider the chain of sublattices
        \begin{equation}\label{eqn:prel_lattice_chain}
            M \oplus M^\perp < L < L^\vee < M^\vee \oplus (M^\perp)^\vee,
        \end{equation}
        from which we obtain the subgroup
        $$
            H \coloneqq L / (M \oplus M^\perp) < (M^\vee \oplus (M^\perp)^\vee )/ (M \oplus M^\perp)= A_M \times A_{M^\perp}.
        $$

        Moreover, we consider the projections
        $$
            p: H \inj A_M \times A_{M^\perp} \surj A_{M^\perp} \qquad\text{and}\qquad
            p': H \inj A_M \times A_{M^\perp} \surj A_{M}.    $$

        Since $M$ is primitive in $L$, the morphism $p$ is injective. Indeed, each $\ell\in L$ can be written as $\ell = rm + sm'$ with $r,s\in \Q$ and $m$ and $m'$ vectors in $M$ and $M^\perp$ respectively. Since $\ell\cdot L\subset \Z$, we obtain that $rm$ is an element of $M^\vee$ and $sm'$ is an element of $(M^\perp)^\vee$. Hence,
        $$
            p([\ell]) = [sm'] = 0 \in A_{M^\perp} \qquad\text{implies}\qquad sm'\in M^\perp.
        $$
        Therefore, the vector $\ell-sm' = rm$ is in $L$. Since $M$ is primitive, this implies $rm\in M$, and therefore $\ell\in M \oplus M^\perp$.
        Analogously, we show that the morphism $p'$ is injective.\\

        By computing the indices from the chain \eqref{eqn:prel_lattice_chain}, we obtain
        \begin{equation}\label{eqn:prel_disc_lattices}
            \disc(M)\disc(M^\perp) = |H|^2\disc(L).
        \end{equation}
        Moreover, the injectivity of $p$ and $q$ implies $|H|\le \disc(M^\perp)$ and $|H|\le \disc(M)$.

        \begin{proposition}[{\cite[Corollary~1.5.2]{Nikulin}}]\label{prop:extend_iso}
            An isometry $g\in O(M^\perp)$ extends to $O(L,M)$ if and only if $\bar g\vert_{p(H)} = \id$.
        \end{proposition}
        \begin{proof} We prove the proposition in the case $M=\Z h$ for some primitive vector $h$ of $L$. In this case, we will denote by $O(L,h)$ the group $O(L, \Z h)$.

            Each isometry $g\in O(h^\perp)$ extends uniquely to an isometry $\tilde g\in O(L_\Q, h)$, defined by $\tilde g(h)=h$ and $\tilde g\vert_{h^\perp_\Q} = g$. The isometry $g$ extends to $O(L,h)$ if and only if $\tilde g(\ell) \in L$   for all $\ell\in L.$

            Observe that each vector $\ell\in L$ can be written as $\ell = r h + s v$, with $v\in h^\perp$ and $r,s\in \Q$. Therefore, $\tilde g(\ell) = r h + s g(v)$.

            Since $\ell\cdot h^\perp\subset \Z$, we obtain that $b = s\div(v)$ is an integer. Notice moreover that $\div(g(v))=\div(v)$ because $g$ in an isometry of $h^\perp$. Hence, we obtain
            $$p([\ell])=[sv] = b \left[\frac{v}{\div(v)}\right]\in A_{h^\perp} \qquad \text{and} \qquad  \bar g(p([\ell])) = [sg(v)]=b \left[\frac{g(v)}{\div(g(v))}\right]\in A_{h^\perp}.$$
            Observe that $\bar g(p([\ell])) = p([\ell])$ if and only if $s(g(v) - v) \in h^\perp$, which is equivalent to
            \begin{equation}\label{eqn:extend_iso}
                        \tilde g(\ell)-\ell \in h^\perp = h^\perp_\Q \cap L.
            \end{equation}
            Since $\ell\in L$ and $\tilde g(\ell)-\ell\in h^\perp_\Q$, equation \eqref{eqn:extend_iso} is equivalent to $\tilde g(\ell)\in L$.
        \end{proof}

        Therefore,
        \begin{equation}\label{eqn:lattice_ext_iso}
            O(L, h)=\{g\in O(h^\perp) \mid \bar{g}\vert_{p(H)}=\id\}.
        \end{equation}


        \begin{proposition}\label{prop:extend_iso_tilde}
            For each primitive vector $h\in L$, there is an inclusion $\widetilde O(h^\perp)\inj \widetilde O(L,h)$.
        \end{proposition}
        \begin{proof}
            Since $p(H)<A_{h^\perp}$, Proposition~\ref{prop:extend_iso} implies that each isometry $g\in \widetilde O(h^\perp)$ extends to an isometry of $O(L,h)$, which we will still denote by $g$. By definition of $\widetilde O(h^\perp)$, the isometry $g$ satisfies $\bar g\vert_{A_{h^\perp}}=\id$. Moreover $g\vert_{\Z h} =\id$, hence $\bar g$ is the identity on $A_{\Z h}\times A_{h^\perp}$, and therefore on $A_L$, which is a subquotient of $A_{\Z h}\times A_{h^\perp}$ (use \eqref{eqn:prel_lattice_chain}).

            Finally, the morphism $\widetilde O(h^\perp)\to \widetilde O(L,h)$ is injective because restriction is a left inverse.
        \end{proof}

        We have the following chain of inclusions
        \begin{equation}\label{eqn:prel_chain_sub}
            \widetilde O(h^\perp)\stackrel{i_1}{\longhookrightarrow} \widetilde O(L,h) \stackrel{i_2}{\longhookrightarrow} \widehat O(L,h) \stackrel{i_3}{\longhookrightarrow} O(L,h) \stackrel{i_4}{\longhookrightarrow} O(h^\perp),
        \end{equation}
        where the index of $i_2$ divides $2$ and the inclusions $i_3$ and $i_3i_2$ define normal subgroups of $O(L,h)$.






\section{Period domains of type IV}\label{sec:periods_domains}
    Let $\Lambda$ be an even indefinite lattice  of signature $(2, n_-)$ such that $n_-\ge 2$. The zero locus of the quadratic form induced on $\Lambda_\C$ is a smooth quadric in $\Proj(\Lambda_\C)$. The open analytic subset
    $$
    \cD_\Lambda = \{[x]\in \Proj(\Lambda_\C) \mid x\cdot x = 0, \quad x\cdot \bar{x}>0 \}
    $$
    of this quadric is a complex manifold, called the \emph{period domain}. One has
        $$
    \cD_\Lambda = \cD_\Lambda^+ \sqcup \cD_\Lambda^-.
    $$
    These two connected components are diffeomorphic, exchanged by complex conjugation.

    \begin{remark}

        If $\Lambda$ is isomorphic to $U\oplus \Lambda'$ for some lattice $\Lambda'$, there exists an isometry $g\in \widetilde O(\Lambda)$ that exchanges the two connected components of $\cD_\Lambda$ (see \cite[Proposition 5.6]{Dolg96}).
    \end{remark}

    Since each isometry of $\Lambda$ acts on $\Proj(\Lambda_\C)$ and preserves $\cD_\Lambda$, we get an action of $ O(\Lambda)$ on $\cD_\Lambda$, which is properly discontinuous {\cite[Remark~6.1.10]{K3Huy}}. Recall the following theorem by Borel--Baily.
    \begin{theorem}[{\cite[Theorem~6.1.13]{K3Huy}}]\label{thm:per_dom_qpvar}
        For all subgroups $\Gamma$ of $O(\Lambda)$ of finite index, the quotient $\cD_\Lambda/\Gamma$ is a normal quasi-projective variety.
    \end{theorem}
    Moreover, if there exists $g\in \Gamma$ that exchanges the two components of $\cD_\Lambda$,  the variety $\cD_\Lambda/\Gamma$ is irreducible.



\section{Ramification divisors of covers of period spaces}\label{sec:ram_divisors}
Let $\Lambda$ be an even indefinite lattice of signature $(2, n_-)$, with $n_-\ge 2$. We fix a subgroup $\Gamma < O(\Lambda)$ of finite index. Theorem~\ref{thm:per_dom_qpvar} implies that
$$
\cD_\Lambda / \Gamma
$$
is a quasi-projective variety, which is irreducible if there exists an element of $\Gamma$ that exchanges the  two connected components of $\cD_\Lambda$. In the following, we will suppose that there exists such an element in $\Gamma$.
Observe that $-\id$ acts trivially on $\cD_\Lambda$, hence on $\cD_\Lambda / \Gamma$. Let $\bar \Gamma$ be the group generated by $\Gamma$ and $-\id$: then $\cD_\Lambda / \Gamma= \cD_\Lambda / \bar\Gamma$.\\

Let us fix another subgroup $O < O(\Lambda)$ of finite index such that
$$
        \Gamma \norm O < O(\Lambda).
$$
Since $-\id$ is in the center of $O(\Lambda)$, the group $\bar \Gamma$ is a normal subgroup of $\bar O$. The group
$$
    G = \bar O/ \bar \Gamma
$$
is a finite group that acts on $\cD_\Lambda / \Gamma$.
\begin{remark}\label{rmk:inv_elements}
    An element $[x]\in \cD_\Lambda / \Gamma$ is fixed by $g\in G$ if and only if there exists an isometry $f\in  O$ such that $[f]=g$ and $x$ is an eigenvector of $f_\C$.\\
    Indeed, if $[x]\in\cD_\Lambda / \Gamma$ is fixed by $g = [f]$, then $[f(x)]=[x]\in \cD_\Lambda / \Gamma$. This means that there exists $\tilde f\in \bar\Gamma$ such that the lines $f_\C(x)\C$ and $ \tilde f_\C(x) \C$ are equal. Replacing $f$ by $\tilde f^{-1} f$ we obtain that $x$ is an eigenvector of $f_\C$. Conversely, by definition, each eigenvector of $f_\C$ defines a line that is fixed by $[f]$.
\end{remark}

 Therefore, the action of $G$ on  $\cD_\Lambda / \Gamma$ is (very)-generally faithful: consider the subset
$$
    X = \bigcup_{g\in G \setminus \{\id\}}\bigcup_{\substack{g=[f] \\ \lambda \in \mathrm{Sp}(f)}} V_\lambda(f) ,
$$
of $\cD_\Lambda$ which is a countable union of closed subvarieties of $\cD_\Lambda$ of codimension greater than or equal to 1. Then, for $x$ not contained in $X$, the stabilizer of $[x]$ in $\cD_\Lambda / \Gamma$ is trivial.

Hence, the action of $G$ on $\cD_\Lambda / \Gamma$ yields a Galois cover  
\begin{equation}\label{eqn:galois_cover_period_dom}
\cover: \cD_\Lambda / \Gamma\map \cD_\Lambda / O
\end{equation}
with Galois group $G$.\\

The varieties $\cD_\Lambda / \Gamma$ and $\cD_\Lambda / O$ are normal varieties. By restricting the morphism $\cover$ to the preimage of the smooth locus of $\cD_\Lambda / O$, purity of the branch locus \cite[Exp. X, Theorem~3.1]{SGA1} implies that the branch locus has codimension 1.
We want to characterize the ramification divisors of the cover $\cover$, namely the irreducible algebraic divisors of $\cD_\Lambda / \Gamma$ contained in the fixed locus of a nontrivial element of $G$.

\subsection{Heegner divisors and reflections}
Let $\beta$ be a vector of $\Lambda$ with $\beta^2<0$. Since the lattice $\beta^\perp$ has signature $(2, n_- - 1)$, we observe that
$$
\cD_{\beta^\perp}= \cD_\Lambda \cap \Proj(\beta^\perp_\C) =  \{[x]\in \Proj(\beta^\perp_\C) \mid x\cdot x = 0,\quad x\cdot \bar{x}>0 \} 
$$
is not empty, and it is a hypersurface of $\cD_\Lambda$. Moreover,
$$
\Hh_{\beta^\perp} = \Im\left(\cD_{\beta^\perp} \map \cD_\Lambda / \Gamma \right)$$
is an algebraic divisor of $\cD_\Lambda / \Gamma$ \cite[Theorem 3.14]{Hasset00}. We observe that $\Hh_{\beta^\perp}$ is irreducible. Indeed, $\cD_{\beta^\perp}$ has 2 connected components exchanged by  complex conjugation, hence they are contained in two different components of $\cD_\Lambda$. This implies that they are identified in the quotient.

\begin{definition}
    A \emph{Heegner divisor} of $\cD_\Lambda / \Gamma$ is a divisor of the form $\Hh_{\beta^\perp}\subset \cD_\Lambda / \Gamma$ for some  $\beta\in\Lambda$ with $\beta^2<0$.
\end{definition}

\begin{lemma}\label{lemma:inv_div_equality}
    Let $\beta$ and $\gamma$ be primitive vectors of $\Lambda$ with negative squares. The divisors $\Hh_{\beta^\perp}$ and $\Hh_{\gamma^\perp}$ of $\cD_\Lambda / \Gamma$ are equal if and only if $\beta$ and $\gamma$ are in the same $\bar\Gamma$-orbit.
\end{lemma}
\begin{proof}
    Let $\pi: \cD_\Lambda\to \cD_\Lambda / \Gamma$ be the canonical projection. For each vector $\beta$ that defines a Heegner divisor, the divisor $\Hh_{\beta^\perp}$ is the image via $\pi$ of the period domain $\cD_{\beta^\perp}\subset\cD_\Lambda$.
    The connected components of $\cD_{\beta^\perp}$ are $\cD_{\beta^\perp}^+ = \cD^+_\Lambda\cap \cD_{\beta^\perp}$ and $\cD_{\beta^\perp}^- = \cD^-_\Lambda\cap \cD_{\beta^\perp}$.

    Observe that
    $$
        \pi^{-1}(\Hh_{\beta^\perp})  = \bigcup_{g\in \bar\Gamma} \cD_{g(\beta)^\perp}.
    $$
    Clearly, if $\beta$ and $\gamma$ are in the same $\bar\Gamma$-orbit, they define the same Heegner divisor.\\

    Conversely, if $\Hh_{\gamma^\perp}=\Hh_{\beta^\perp}$, then $\cD_{\gamma^\perp}$ is contained in $\pi^{-1}(\Hh_{\beta^\perp})$ and in particular
    $$
                \cD_{\gamma^\perp}^+ \subset \pi^{-1}(\Hh_{\beta^\perp})\cap \cD^+_\Lambda = \bigcup_{g\in \bar\Gamma} \cD_{g(\beta)^\perp}^+.
    $$
    Since $\cD_{\gamma^\perp}^+$ is irreducible, there exists $g\in\bar\Gamma$ such that $\cD_{\gamma^\perp}^+ = \cD_{g(\beta)^\perp}^+$. As complex conjugation exchanges $\cD^+_\Lambda$ and $\cD^-_\Lambda$, we obtain $\cD_{\gamma^\perp} = \cD_{g(\beta)^\perp}$.

    We show that this implies $g(\beta)^\perp=\gamma^\perp$. If not, the closed subvariety $\Proj(g(\beta)^\perp_\C)\cap \Proj(\gamma^\perp_\C)$ is a hypersurface of $\Proj(\gamma^\perp_\C)$ that contains $\cD_{\gamma^\perp}$. Therefore, it contains its closure $\{[x]\in \Proj(\gamma^\perp_\C)\mid x^2 =0\}$, which is an irreducible quadric, hence not contained in any hypersurface.
    Since $\gamma$ and $\beta$ are primitive this implies $g(\beta)=\pm \gamma$ and therefore $\gamma$ and $\beta$ are in the same $\bar\Gamma$-orbit.
\end{proof}


For each vector $\beta\in \Lambda$ with $\beta^2\neq 0$, the reflection with respect to $\beta^\perp$ in $\Lambda_\Q$ is given by the formula
$$
\forall x\in \Lambda \qquad r_\beta(x) = x - \frac{2x\cdot\beta}{\beta^2}\beta.
$$

When $\beta$ is primitive, $r_\beta$ is in $O(\Lambda)$ if and only if
$
\beta^2\mid 2 \div(\beta).
$
\begin{definition}\label{def:reflection}
    A primitive vector $\beta\in \Lambda$ with $\beta^2<0$ \emph{defines a nontrivial reflection in $G$} if $\beta^2\mid 2\div(\beta)$, the reflection $r_\beta$ is in the group $\bar O$, and $[r_\beta]\in G$ is nontrivial.
\end{definition}

If $\beta$ defines a nontrivial reflection in $G$, the Heegner divisor $\Hh_{\beta^\perp}\subset \cD_\Lambda /\Gamma$ is contained in the fixed locus of $r_\beta$.
\subsection{The ramification divisors of $\cover: \cD_\Lambda / \Gamma \to \cD_\Lambda / O$}
The next theorem generalizes \cite[Proposition~3.8]{stellari2008}, following {\cite[Corollary~2.13]{GritsenkoRefl2007}}: in situation \eqref{eqn:galois_cover_period_dom}, we show that the divisorial components of the ramification of $\cover$ are Heegner divisors associated with nontrivial reflections in $G$.
\begin{theorem}\label{thm:divisori_ref}
    Let $\Lambda$ be an even lattice of signature $(2, n_-)$ with $n_-\ge 2$, and let $\Gamma$ and $O$ be subgroups of finite index of $O(\Lambda)$ such that $\Gamma \norm O$, and $\Gamma$ contains an isometry that exchanges the two connected components of $\cD_\Lambda$. We set $G\coloneqq \bar O /\bar \Gamma$.\\
    An irreducible divisor $D\subset \cD_\Lambda / \Gamma$ is contained in the fixed locus of a nontrivial element $g$ of $G$ if and only if it is a Heegner divisor $\Hh_{\beta^\perp}$, where $\beta$ is primitive with $\beta^2<0$, defines a nontrivial reflection in $G$ and $g = [r_\beta]$.\\
    Moreover, each irreducible divisor of $\cD_\Lambda / \Gamma$ is contained in the fixed locus of at most one nontrivial element $g\in G$.
\end{theorem}
\begin{proof}
    Remark~\ref{rmk:inv_elements} implies that the set of points of $\cD_{\Lambda}/\Gamma$ fixed by $g\in G$ is
    $$
    \mathrm{Fix}(g) = \pi\left(\bigcup_{[f]=g} \bigsqcup_{\lambda \in \mathrm{Sp}(f_\C)} \Proj(V_\lambda(f_\C))\cap \cD_\Lambda \right),
    $$
    where $\pi: \cD_\Lambda \to \cD_\Lambda / \Gamma$ is the canonical projection and $V_\lambda(f_\C)$ is the eigenspace of $f_\C$ relative to the eigenvalue $\lambda$.\\

    Observe that if $\mathrm{Fix}(g)$ contains an irreducible divisor $D$, there exists an isometry $f\in \bar O$ with $[f]=g$ and an eigenvalue $\lambda$ of $f_\C$ such that $V_\lambda(f_\C)$ has codimension $1$.
    Indeed, $D$ has codimension 1 in $\cD_\Lambda$ and
    \begin{align*}
    D &= \pi\left(\pi^{-1}(D)\cap \bigcup_{[f]=g} \bigsqcup_{\lambda \in \mathrm{Sp}(f_\C)} \Proj(V_\lambda(f_\C))\cap \cD_\Lambda  \right) \\
    &= \bigcup_{\substack{[f]=g \\ \lambda \in \mathrm{Sp}(f_\C)} } \pi \left(\pi^{-1}(D)\cap \Proj(V_\lambda(f_\C)) \right),
    \end{align*}
    where the union is over a countable set, as $\Gamma$ is countable.  Hence at least one of the pieces $\pi(\pi^{-1}(D)\cap \Proj(V_\lambda(f_\C)))$ has codimension 1, therefore so has $\pi^{-1}(D)\cap \Proj(V_\lambda(f_\C))$ and this implies the claim.\\

    Moreover, since D is irreducible, we obtain
    $$
    D =  \pi(\Proj(V_\lambda(f_\C))\cap \cD_\Lambda).
    $$

    For each real operator, the eigenspace relative to an eigenvalue $\lambda$ has the same dimension as the eigenspace relative to $\bar\lambda$. Since $f_\C$ is a real operator and an isometry, and the codimension of $V_\lambda(f_\C)$ is 1, it follows that $\lambda = \pm 1$. Up to changing $f$ into $-f$, we can suppose $\lambda = 1$.  \\

    Since $\mathrm{codim}(V_1(f_\Q)) = \mathrm{codim}(\ker(\id - f_\Q)) =\mathrm{codim}(V_1(f_\C)) = 1$, there exists $\beta\in \Lambda$ primitive such that
    $$
    V_1(f_\Q) = \beta^\perp \text{ and } f_{\Q|_{\beta^\perp}}= \id.
    $$

    Observe moreover that $\beta^2 < 0$. Indeed if $x\in \cD_\Lambda \cap \Proj(V_1(f_\C))$, then $f_\C(\bar{x})=\bar{x}$, so $P = \mathrm{Re}(x)\R \oplus \mathrm{Im}(x)\R \subset V_1(f_\C)$. As $P$ is positive definite, it follows that $n_+(V_1(f_\C))= 2$, hence $\beta^2 < 0$.
    Hence $f$ satisfies $f\vert_{\beta^\perp}= \id$ and $f(\beta)=-\beta$, namely $f_\Q$ is the reflection with respect to $\beta$ and $[r_\beta] = [f] = g  \in G$ is nontrivial. Therefore $D$ is a Heegner divisor and $\beta$ defines a nontrivial reflection in $G$.\\

    Suppose there exists $g$ and $g'$ in $G$ such that $D\subset \mathrm{Fix}(g)\cap \mathrm{Fix}(g')$. We have proved that there exist vectors $\beta$ and $\gamma$ that define nontrivial reflections such that $D = [\Hh_{\beta^\perp}]=[\Hh_{\gamma^\perp}]$ with $g=[r_\beta]$ and $g' = [r_\gamma]$. Lemma~\ref{lemma:inv_div_equality} implies that $\gamma = g\beta$ for some $g\in \bar\Gamma$. Thus, since $r_{g\beta}=g r_\beta g^{-1}$, it follows that $g' = [r_{g\beta}]=[r_\beta]=g$.
\end{proof}

Hence, the ramification divisors of the morphism $\cover:\cD_\Lambda / \Gamma \to \cD_{\Lambda}/O$ are parametrized by the $\bar\Gamma$-orbits of vectors $\beta\in \Lambda$ that define nontrivial reflections in $\bar O /\bar \Gamma$.\\

We notice that given $g\in G$, the fixed locus $\mathrm{Fix}(g)$ may contain several divisorial components, namely we could have $g=[r_\beta]$ for several vectors $\beta$ that are not in the same $\bar \Gamma$-orbit.

\section{Groups of isometries of some lattices}\label{sec:isometries_lattices_nomality}
    \def\cm{a} 
    \def\vm{m} 
    \def\ck{b} 
    \def\cl{c} 
    We now apply the results of Section~\ref{sec:ram_divisors} to moduli spaces of polarized hyper-Kähler manifolds of K3$^{[m]}$-type.
    As in the introduction we have a cover

    \begin{equation}\label{eqn:cover_HK}
    \cover: \cD_{h^\perp}/\widehat O(\Lambda_{K3^{[m]}}, h^\perp) \map \cD_{h^\perp}/ O(h^\perp),
    \end{equation}
    where $\cD_{h^\perp}/\widehat O(\Lambda_{K3^{[m]}}, h^\perp)$ is the period space of polarized hyper-Kähler manifolds of K3$^{[m]}$-type with polarization of type $h$. In order to apply Theorem~\ref{thm:divisori_ref} to the cover \eqref{eqn:cover_HK},
    we need to study the normality of the subgroup $\widehat O(\Lambda_{K3^{[m]}}, h)$ of $O(h^\perp)$.\\



    We consider a slightly more general situation.  Given a positive integer $t$, let $L_{2t}$ be the even lattice
    \begin{equation}\label{eqn:lattice_2t}
        L_{2t} = M \oplus U \oplus \Z\ell,
    \end{equation}
    where $M$ is an even unimodular lattice and $\ell^2 = -2t$.
    The discriminant group $A_{L_{2t}}$ is a cyclic group of order $2t$ generated by $\ell_*$.  The lattice $\Lambda_{K3^{[m]}}$ is a lattice of type $L_{2(m-1)}$.\\

    Let $h$ be a primitive  vector of $L_{2t}$ of square $2d>0$ and divisibility $\gamma$. Recall from~\eqref{eqn:prel_chain_sub} the chain of subgroups
        \begin{equation}\label{eqn:inv_div_chain_sub}
            \widetilde O(h^\perp)\stackrel{i_1}{\longhookrightarrow} \widetilde O(L_{2t},h) \stackrel{i_2}{\longhookrightarrow} \widehat O(L_{2t},h) \stackrel{i_3}{\longhookrightarrow} O(L_{2t},h) \stackrel{i_4}{\longhookrightarrow} O(h^\perp),
        \end{equation}
        where we can describe the group $O(L_{2t}, h)$ as in \eqref{eqn:lattice_ext_iso}.\\



    We study the lattice $h^\perp$ and the group $A_{h^\perp}$, following \cite{gritsenko2010}. Moreover we describe the image of $\widehat O(L_{2t}, h)$ in $O(A_{h^\perp})$ in some cases. More precisely,
    \begin{itemize}
        \item{in Proposition \ref{prop:HK_struct_perp}} we describe the lattice $h^\perp$,
        \item{denoting by $\omega$ the $\gcd(\frac{2t}{\gamma}, \gamma)$,
                \begin{itemize}
                    \item we compute the discriminant group $A_{h^\perp}$ for $\omega=1$ (Proposition~\ref{prop:dec_disc_group});
                    \item we describe the image of the group $\widehat O(L_{2t}, h)$ in $O(A_{h^\perp})$ under the morphism $r: O(h^\perp)\to O(A_{h^\perp})$ introduced in Theorem~\ref{thm:Nik_surjMor} (Proposition~\ref{prop:Beri} for $t=1$ or $\gamma > 2$ and Proposition~\ref{prop:HK_restr_hat_group} for $\gamma \in \{1,2\}$ and $\omega=1$);
                    \item we discuss the normality of $\widehat O(L_{2t}, h)$ in $O(h^\perp)$ and show that, if $(t,d)=1$, then $\widehat O(L_{2t}, h)\norm O(h^\perp)$ (Corollary~\ref{cor:normality_coprime_case}).
                \end{itemize}
                }
    \end{itemize}

    \subsection{The lattice $h^\perp$}\label{subsec:hperp}
    The primitive vector $h$ in the lattice $L_{2t}$ can be written as 
    $$
        h = am + c\ell
    $$
    where $m\in M\oplus U$ is primitive and $a, c$ are coprime integers. The divisibility of $h$ is $\gamma=(a, 2tc)=(a,2t)$. In particular $\gamma\mid 2t$ and we can write $a = \gamma a_1$ for some $a_1\in \Z$. Observe moreover that, since $\gamma\mid a$, we have $(c,\gamma)=1$. Finally,
    \begin{equation}\label{eqn:HK_classDisc}
                h_* = \left[\frac{h}{\div(h)}\right] = \bar c \frac{2t}{\gamma} \ell_* \in A_{L_{2t}},
    \end{equation}
    of order $\gamma$ in $A_{L_{2t}}$. The class $\bar c\in \Z/\gamma \Z$ is uniquely determined by $h_*$.

    By computing the square of $h$, we obtain
    $
        2d = h^2= \gamma^2 a_1^2 m^2 -2tc^2,
    $
    where $m^2$ is an even integer. Thus, the quotient $\frac{d+tc^2}{\gamma^2}$ is an integer which we denote by $b$.

    Given a standard basis $(e, f)$ of $U$, we consider the vector
    \begin{equation}\label{eqn:HK_vector_form}
        \tilde h = \gamma(e + bf) + c\ell.
    \end{equation}
    Its divisibility is $(\gamma, 2tc)$, which is $\gamma$ since $\gamma\mid 2t$, its square is $2\gamma^2b -2tc^2 = 2d$, and
    $$
        \tilde h_* = \left[\frac{\tilde h}{\div(\tilde h)} \right] = \bar c\frac{ 2t}{\gamma} \ell_* = h_* \in A_{L_{2t}}.
    $$
    Since by Eichler's Lemma, the $\widetilde O(L_{2t})$-orbit of $h$ is determined by $h^2$ and $h_*$, and we are only interested in the $O(L_{2t})$-orbit of $h$, we can suppose that the vector $h$ is of the form \eqref{eqn:HK_vector_form}.\\

    Note that the element $h_*$ of $A_{L_{2t}}$ (see \eqref{eqn:HK_classDisc}) is determined by $c \pmod{\gamma}$. If $c = \gamma n + c'$, then, for $b' = b + tn^2 - \frac{2t}\gamma nc$, the vector
    $$
        h'=\gamma(e + b'f) + c'\ell
    $$
    has square
    $$2\gamma^2b' - 2t c'^2 = 2\gamma^2b + 2t\gamma^2 n^2 - 4t\gamma nc - 2t(\gamma^2n^2 - 2\gamma nc +c^2) = 2\gamma^2b -2tc^2 = 2d,$$
    and $h'_*=h_*$ in $A_{L_{2t}}$.

    So we may always assume that $0\le c < \gamma$ and $(c, \gamma) =1$.

    \begin{remark}\label{rmk:gamma_t_d} If $h$ is a primitive vector of $L_{2t}$ of divisibility $\gamma$ and square $2d$, then
    \begin{equation}\label{eqn:remark_gamma_t_d}
    \gamma^2\mid d + tc^2.
    \end{equation}
    Therefore, in general not all pairs $(2d, \gamma)$ can be realized as $(h^2, \div(h))$ for some primitive vector $h\in L_{2t}$.\\
    For instance, if $\gamma =1$ the condition~\eqref{eqn:remark_gamma_t_d} is always verified. If $\gamma =2$, then $c$ is necessarily $1$ and $d$ must verify $d + t\equiv 0\pmod{4}.$
    Hence, for $\gamma \in \{0,1\}$, the orbit of $h$ is uniquely determined by $\gamma, d$. In these cases, we denote by $^{[m]}\Mm_{2d}^{(\gamma)}$ and $^{[m]}\Pp_{2d}^{(\gamma)}$ the moduli space and the period space of hyper-Kähler manifolds of K3$^{[m]}$-type with polarisation of square $2d$ and divisibility $\gamma$.

    \end{remark}
    \begin{proposition}[{\cite[Proposition~3.6.(\rm{iv})]{gritsenko2010}}]\label{prop:HK_struct_perp}
         Let $h$ be a primitive vector of $L_{2t}$ of square $2d$ and divisibility $\gamma$, and let $c$ be the integer such that $0\le c < \gamma$ and $(c,\gamma)=1$ defined in \eqref{eqn:HK_vector_form}. Then,
        $$
            h^\perp = M\oplus \begin{pmatrix}
                                                                    - \dfrac{2d + 2c^2t}{\gamma^2} & c\dfrac{2t}{\gamma}\\[10pt]
                                                                    c\dfrac{2t}{\gamma} & -2t
                                                                    \end{pmatrix}.
        $$
        In particular, up to isometries of $L_{2t}$, we can suppose $h = \gamma(e+bf) + c\ell$.  The vectors
        $$
            h_1 = e-bf \quad\text{and}\quad h_2 =c\frac{2t}\gamma f + \ell
        $$
        form a basis of the non unimodular part of $h^\perp$.
    \end{proposition}


    \begin{remark}\label{rmk:gamma1} If $\gamma = 1$, then $h_* = 0 \in A_{L_{2t}}$. Therefore $c=0$, the lattice $B$ is diagonal, and there is an isomorphism
    $$
        h^\perp \simeq M \oplus \Z(-2d)\oplus \Z(-2t),
    $$
    where, if we take $h=e+df$, a basis for the non unimodular part of $h^\perp$ is given by $k = e-df$ and $\ell$.
    \end{remark}

    \subsection{The groups $\widetilde O(L_{2t}, h)$ and $\widehat O(L_{2t}, h)$}
    From \eqref{eqn:lattice_ext_iso}, we have
    $$
        O(L_{2t}, h) = \{g\in O(h^\perp) \tq \bar g\vert_{p(H)}= \id \},
    $$
    where $H$ is the group $L_{2t}/(\Z h \oplus h^\perp)$ and $p$ is the projection $H\inj A_{\Z h}\times A_{h^\perp} \to  A_{h^\perp}$.\\

    Given $h_1$ and $h_2$ as in Proposition~\ref{prop:HK_struct_perp}, each $v\in L_{2t} = M \oplus U \oplus \Z \ell$ can be written as ${v= m + a_1 h_1+a_2h_2+df}$, with $m\in M$ and $a_1, a_2, d\in \Z$, and such a vector $v$ is orthogonal to $h$ if and only if $d=0$. Therefore we obtain
    $$
        H = L_{2t}/(\Z h \oplus h^\perp) = \langle[f]\rangle.
    $$

    We describe the image $p(H)$, which is generated by $p([f])$.
    The vector
    \begin{equation}\label{eqn:k_1_vecor_disc_group}
        k_1 = \frac{\gamma}{2d} h - f
    \end{equation}
    is in $(h^\perp)^\vee$. Indeed, we can compute $k_1\cdot h_1 = -1$ and $k_1\cdot h_2 = k_1\cdot M = 0$. Notice, moreover, that $p([{f}]) = -\bar k_1$. Hence, the group $p(H)$ is generated by $\bar k_1 \in A_{h^\perp}$.\\

    Therefore, the groups $\widetilde O(L_{2t}, h) = \widetilde O(L_{2t})\cap O(L_{2t}, h)$ and $\widehat O(L_{2t}, h) = \widehat O(L_{2t})\cap O(L_{2t}, h)$ can be described as
    \begin{equation}\label{eqn:HK_tilde_grp}
        \widetilde O(L_{2t}, h)  = \{g\in O(h^\perp) \mid \bar g(\bar k_1) = \bar k_1 \in A_{h^\perp}\ \ \text{and} \ \ \bar g(\ell_*) = \ell_*\in A_{L_{2t}} \}
    \end{equation}
    and
    \begin{equation}\label{eqn:HK_mon_grp}
        \widehat O(L_{2t}, h)  = \{g\in O(h^\perp) \mid \bar g(\bar k_1) = \bar k_1 \in A_{h^\perp}\ \ \text{and} \ \ \bar g(\ell_*) = \pm \ell_*\in A_{L_{2t}} \}.
    \end{equation}

    \subsection{The discriminant group $A_{h^\perp}$} We study the discriminant group $A_{h^\perp}$. From Equation~\eqref{eqn:prel_disc_lattices} it follows that
    \begin{equation}\label{eqn:HK_cardinality_disc_perp}
        \disc(\Z h) \disc(h^\perp) = |H|^2 \disc(L_{2t}),
    \end{equation}
    where $\disc(\Z h) = 2d, \disc(L_{2t}) =2t$ and $\disc(h^\perp)=|A_{h^\perp}|$.\\

    Observe that the element $\bar k_1$ defined in \eqref{eqn:k_1_vecor_disc_group} has order $\frac{2d}\gamma$ in $A_{h^\perp}$. Indeed, given an integer $n\in \Z$, the vector $nk_1$ is in ${h^\perp} = h^\perp_\Q\cap L_{2t}$ if and only if $n\frac{\gamma}{2d} h\in L_{2t}$, hence if and only if $n\frac{\gamma}{2d}\in \Z$.

    We showed in Section~\ref{sec:prelim_iso} that the morphism $p$ is injective, hence we obtain
    $$
        |H| = |p(H)| = \frac{2d}\gamma.
    $$

    From Equation~\eqref{eqn:HK_cardinality_disc_perp} we get
    $$
       2d \cdot |A_{h^\perp}| = 2t \left(\frac{2d}\gamma\right)^2,
    $$
    from which we obtain that $A_{h^\perp}$ is an abelian group of cardinality $\frac{2d}\gamma\frac{2t}\gamma$.\\

    Finally, note that
    $$\omega\coloneqq\left(\frac{2t}\gamma, \frac{2d}\gamma, \gamma\right)=\left(\frac{2t}\gamma, \gamma\right).$$
    Indeed, from \eqref{eqn:HK_vector_form} we can suppose that $h = \gamma(e+bf) + c\ell$, and therefore
    \begin{equation}\label{eqn:HK_square_div}
        2d = h^2 = 2b\gamma^2 -2tc^2 =\gamma\left(2b\gamma -\frac{2t}\gamma c^2\right).
    \end{equation}
    Hence $(\frac{2t}\gamma, \gamma)\mid \frac{2d}\gamma$.
    The next result shows that, for $\omega = 1$, the structure of the discriminant group $A_{h^\perp}$ is particularly simple.
    \begin{proposition}[{\cite[Proposition~3.12]{gritsenko2010}}]\label{prop:dec_disc_group}
        Let $h\in L_{2t}$ be a primitive vector with $h^2=2d$ and $\div(h)=\gamma$.
        If $\omega = (\frac{2t}\gamma, \gamma) = 1$, there exists an isometry
        \begin{equation}\label{eqn:disc_grp_omega1}
            A_{h^\perp} \simeq \Z / \tfrac{2d}\gamma \Z \times \Z / \tfrac{2t}\gamma \Z
        \end{equation}
        such that the subgroup $p(H)< A_{h^\perp}$ corresponds to the factor $\Z/ \frac{2d}\gamma \Z$ and the intersection form on $A_{h^\perp}$ is defined by $q(1,0)=-\frac{\gamma^2}{2d}$ and $q(0,1)=-\frac{\gamma^2}{2t}$.
    \end{proposition}
    The key point is that, for $\omega =1$, the classes of the vectors
    $$
    k_1 = \frac{\gamma}{2d} h - f \quad \text{and}\quad k_2 = c f + \frac{\gamma}{2t} \ell
    $$
    generate the discriminant group $A_{h^\perp}$. 

    \subsection{Normality of $\widehat O(L_{2t}, h)$ in $O(h^\perp)$}
        As in \eqref{eqn:prel_chain_sub}, we consider the chain of subgroups
        \begin{equation}\label{eqn:inv_div_chain_sub_omega1}
            \widetilde O(h^\perp)\stackrel{i_1}{\longhookrightarrow} \widetilde O(L_{2t},h) \stackrel{i_2}{\longhookrightarrow} \widehat O(L_{2t},h) \stackrel{i_3}{\longhookrightarrow} O(L_{2t},h) \stackrel{i_4}{\longhookrightarrow} O(h^\perp),
        \end{equation}
        where the inclusions $i_4i_3i_2i_1$ and $i_3i_2$ define normal subgroups.

        We want to understand when $\widehat O(L_{2t}, h)$ is a normal subgroup of $ O(h^\perp)$.  A summary of the results that follow  can be found in Remark~\ref{rmk:HK_summary_normality_subgroup}.

        \begin{proposition}[{\cite[Lemma 3.5, Proposition 3.6]{BeriBarros}}]\label{prop:Beri}
        The inclusion $i_1$ is trivial, and the inclusion
        $$
            i_2: \widetilde O(L_{2t},h) \longhookrightarrow \widehat O(L_{2t},h)
        $$
        has index $1$ if $t=1$ or $\gamma>2$, index $2$ otherwise.
        \end{proposition}
     Hence, if $t=1$ or $\gamma>2$, the group $\widehat O(L_{2t}, h)$ is equal to $\widetilde O(h^\perp)$, and thus it is a normal subgroup of $O(h^\perp)$.

    \begin{proposition}\label{prop:HK_restr_hat_group}
        If $\omega=1$ and $\gamma$ is $1$ or $2$, one has
        $$\widehat O(L_{2t}, h) = r^{-1}\left(\{\id, s\}\right),
        $$
        where the morphism $O(h^\perp)\stackrel{r}{\to} O(A_{h^\perp})$ was defined in Theorem~\ref{thm:Nik_surjMor}, and $s$ is the element of $O(A_{h^\perp})$ acting as $\begin{pmatrix} 1& 0 \\ 0 & -1\end{pmatrix}$ in the decomposition \eqref{eqn:disc_grp_omega1}.
    \end{proposition}
    \begin{proof}

        We show that, if $\gamma$ is either $1$ or $2$, the isometry $s\in O(A_{h^\perp})$ is in the image $r(\widehat O(L_{2t}, h))$. The integer $c$, being prime to $\gamma$ and determined modulo $\gamma$, is $0$ for $\gamma = 1$ and $1$ for $\gamma = 2$. In this cases, the vector $y = ctf + \ell$ defines a reflection $r_y$ on $L_{2t}$ such that $r_y(\ell_*)=-\ell_*\in A_{L_{2t}}$. Indeed, the vector $y$ has square $-2t$ and divisibility equal to $(ct, 2t)$. Hence, it defines a reflection on $L_{2t}$, because $-2t\mid 2(ct, 2t)$, and $$
            r_y\left(\frac{\ell}{2t}\right)=\frac{\ell}{2t} - \frac{2}{2t}\frac{y\cdot \ell}{y^2} y = \frac{\ell}{2t} - \frac{2}{2t}(ctf+l) \equiv -\frac{\ell}{2t} \pmod{L_{2t}}.
        $$

        Proposition~\ref{prop:Beri} shows that if $g\in \widetilde O(L_{2t}, h)$, then $\bar{g}\coloneqq r(g)=\id$.
        Let $g$ be an isometry of $\widehat O(L_{2t}, h)\setminus \widetilde O(L_{2t}, h)$, namely $g$ satisfies $\bar{g}(\bar k_1) = \bar k_1$ and $g(\frac{\ell}{2t})=-\frac{\ell}{2t}+m$ for some $m\in L_{2t}$ (see \eqref{eqn:HK_mon_grp}). We show that $\bar g(\bar k_2) = -\bar k_2$; that implies $\bar g=s$.\\

        For $\gamma=1$, the integer $c$ is $0$, and the vector $\frac{\ell}{2t}$ is orthogonal to $h= e+bf$. Hence, since $g$ is an isometry of $L_{2t}$ that fixes $h$, the vector $g(\frac{\ell}{2t})$ is orthogonal to $g(h)=h$, and that implies $m\in h^\perp$. Moreover, in this case, $k_2 = \frac{\ell}{2t}$, thus we have $\bar g(\bar k_2) = -\bar k_2$.\\

        For $\gamma=2$, the integer $c$ is $1$, and we have $h= 2(e+bf)+\ell$. From $g(k_1)\equiv k_1 \pmod{h^\perp}$, we obtain
        $$
           \frac{\gamma}{2d} h - f = k_1 \equiv g(k_1) =  g\left(\frac{\gamma}{2d} h - g(f)\right) = \frac{\gamma}{2d} h - g(f) \pmod{h^\perp},
        $$
        and therefore $g(f)\equiv f\pmod{h^\perp}$. From $g(h)=h$ follows
        $$
             2(e+bf) +\ell = h = g(h) \equiv  2(g(e)+ bf) - \ell +2tm \pmod{h^\perp},
        $$
        which implies $2tm\equiv 2e + 2\ell - 2g(e) \pmod{h^\perp}$. Observe that the vector $h_1 = e-bf$ is in the lattice $h^\perp$, therefore $g(h_1)$ is in $h^\perp$ too, and hence we have $g(e)\equiv bg(f)\equiv bf\pmod{h^\perp}$. Thus, we obtain
        $$2tm\equiv2e + 2\ell - 2g(e) \equiv 2e + 2\ell - 2bf\equiv  2\ell\equiv 2(\ell+tf)-2tf \equiv -2tf \pmod{h^\perp},$$
        where we used that $y= \ell+tf$ is orthogonal to $h$. Therefore, the vector $m+f$ is an integral vector that belong to the lattice $h^\perp$, thus we have $m\equiv -f\pmod{h^\perp}$.

        Finally, by computing the image of $k_2 = f +\frac{\ell}{t}$, we have
        $$
            g(k_2) \equiv f -\frac{\ell}{t} + 2m \equiv f -\frac{\ell}{t} - 2f \equiv -k_2 \pmod{h^\perp}.
        $$
        As explained above, this proves $\bar g=s$.
    \end{proof}

    Theorem~\ref{thm:Nik_surjMor} implies that, if the unimodular part $M$ of $h^\perp$ has rank at least 2 (which is the case for $L_{2t}=\Lambda_{K3^{[t+1]}}$), the morphism $O(h^\perp)\stackrel{r}{\to} O(A_{h^\perp})$ is surjective. In particular, in this case and under the hypotheses of the previous proposition, the group $\widehat O(L_{2t}, h)$ is normal in $O(h^\perp)$ if and only if the group

    $$
        K = \{\id, s\}
    $$
    is a normal subgroup of $O(A_{h^\perp})$.


    \begin{example} The group $K$ is not always a normal subgroup of $O(A_{h^\perp})$. For example, for $t=9$, $\gamma = 2$ and $d=15$, the group $A_{h^\perp}$ is of the form
    $$
        A_{h^\perp} = \Z/15\Z \times \Z/9\Z
    $$
    with quadratic form defined by $q(1,0) \equiv -\frac2{15} \pmod{2\Z}$ and $q(0,1)\equiv-\frac{2}{9}\pmod{2\Z}$ (see Proposition~\ref{prop:dec_disc_group}).\\ The morphism $g$ defined by the matrix $\begin{pmatrix} 1& 10 \\ 6 & 2\end{pmatrix}$ is an isometry of $A_{h^\perp}$: indeed it is an involution and for each $(x,y)\in A_{h^\perp}$, we can compute
    $$
        q(g(x,y)) = q((x + 10y, 6x + 2y)) \equiv -\frac{2}{15} x^2 -\frac{2}{9} y^2= q(x,y) \pmod{2\Z}.
    $$
    However,
    $$
     \begin{pmatrix} 1& 10 \\ 6 & 2 \end{pmatrix}^{-1}\begin{pmatrix} 1& 0 \\ 0 & -1 \end{pmatrix}\begin{pmatrix} 1& 10 \\ 6 & 2\end{pmatrix} = \begin{pmatrix} 1& 5 \\ 3 & 2\end{pmatrix}.$$
     Therefore, in this case, $K$ is not a normal subgroup of $O(A_{h^\perp})$.
    \end{example}

    \begin{lemma}\label{lemma:norm_group}
    Let $A$ be the group
    $$
        A = \Z/\tfrac{2d}{\gamma}\Z \times \Z/\tfrac{2t}{\gamma}\Z.
    $$
    If $t$ and $d$ are coprime integers and $\gamma$ is either $1$ or $2$, then $K = \{\id, s\}$ is a normal subgroup of $\mathrm{Aut}(A)$.
    \end{lemma}
    \begin{proof}
    If $\gamma=2$, then $A = \Z/d\Z \times \Z/t\Z$ with $(t,d)=1$. Hence $A\simeq \Z/td\Z$ and from \cite[Lemma~3.6.1]{Scat87}, we see that $\Aut(A)$ is abelian, and in particular it follows $K\norm \Aut(A)$.\\
    We now consider the case $\gamma = 1$, hence $A = \Z/2d\Z \times \Z/2t\Z$.
    Let $g = \begin{pmatrix} a& e \\ b & f\end{pmatrix}$ be an automorphism of $A$. The order of $(a,b)$ in $A$ is equal to the order $2d$ of $(1,0)$ in $A$, hence we obtain
    $$
        \lcm\left(\frac{2d}{(a,2d)}, \frac{2t}{(b,2t)}\right)=\ord{a,b}{A}=2d,
    $$
    hence $2t\mid 2d(b, 2t)$.
    Since $t$ and $d$ are coprime, it follows that $t\mid b$ and we can write $b = tb'$. Analogously, we can write $e = de'$. Therefore, we can compute
    $$
        gs = \begin{pmatrix} a& de' \\  tb'& f\end{pmatrix}\begin{pmatrix} 1& 0 \\ 0 & -1\end{pmatrix} = \begin{pmatrix} a&  -de'\\ tb' & -f\end{pmatrix}$$
        and
    $$sg = \begin{pmatrix} 1& 0 \\ 0 & -1\end{pmatrix}\begin{pmatrix} a&  de'\\ tb' & f\end{pmatrix} = \begin{pmatrix} a&  de'\\ -tb' & -f\end{pmatrix}.
    $$
    Since $tb'\equiv -tb'\pmod{2t}$ and $de'\equiv -de'\pmod{2d}$, we obtain $sg=gs$ and hence $g^{-1}sg = s$.

    In both cases, we proved $K\norm\mathrm{Aut}(A)$.
     \end{proof}

    Since $O(A_{h^\perp})$ is a subgroup of $\mathrm{Aut}(A_{h^\perp})$ that contains $K$, we obtain the following corollary.
    \begin{corollary}\label{cor:normality_coprime_case}
    Let $h\in L_{2t}$ be a primitive vector of square $2d$ such that $(t,d)=1$. The group $\widehat O(L_{2t}, h)$ is a normal subgroup of $O(h^\perp)$.
    \end{corollary}
    \begin{proof}
    Since the divisibility $\gamma$ of $h$ divides $(2t, 2d)$, if $t$ and $d$ are coprime, it follows that $\gamma$ is either $1$ or $2$. Moreover, from $(t,d)=1$ we obtain $\omega = (\frac{2t}{\gamma}, \frac{2d}{\gamma}, \gamma) = 1$. Proposition~\ref{prop:dec_disc_group} provides an isomorphism
    $$
    A_{h^\perp} \simeq \Z/\tfrac{2d}{\gamma}\Z \times \Z/\tfrac{2t}{\gamma}\Z
    $$
    where $t$, $d$ and $\gamma$ satisfy the hypotheses of Lemma~\ref{lemma:norm_group}. Therefore, the group $K$ is normal in $\mathrm{Aut}(A_{h^\perp})$ and hence in $O(A_{h^\perp})$. Since $K =r^{-1}(\widehat O(L_{2t}, h))$, we obtain that $\widehat O(L_{2t}, h)$ is a normal subgroup of $O(h^\perp)$.
    \end{proof}

    \begin{remark}\label{rmk:HK_summary_normality_subgroup}
    To sum up, we have proved the following.
    \begin{itemize}
        \item{If $t=1$ or $\gamma > 2$, then $\widehat O(L_{2t}, h) = \widetilde O(h^\perp)$ (Proposition~\ref{prop:Beri}), hence $\widehat O(L_{2t}, h)$ is a normal subgroup of $O(h^\perp)$.}
        \item{if $t>1$, with $\omega =1$ and $\gamma \in\{1, 2\}$, then $\widehat O(L_{2t}, h) =r^{-1}\left(\bigg\{\id, s=\begin{pmatrix} 1& 0 \\ 0 & -1\end{pmatrix}\bigg\}\right).$ When moreover $(t,d)=1$,
    the group $\widehat O(L_{2t}, h) $ is a normal subgroup of $O(h^\perp)$.}
    \end{itemize}
    \end{remark}



\section{Vectors $\beta$ that define nontrivial reflections}
    If $\widehat O(\Lambda_{K3^{[m]}}, h)\norm O(h^\perp)$,  the cover $\cover$ described in \eqref{eqn:cover_HK} is a ramified Galois cover of group
$$
    G \simeq O(h^\perp)\big/\langle\widehat O(\Lambda_{K3^{[m]}}, h), -\id \rangle,
$$
By Remark~\ref{rmk:HK_summary_normality_subgroup} (used for $t=-m-1$), we obtain that
\begin{itemize}
    \item{if $m=2$, or $\gamma > 2$, then $G\simeq O(A_{h^\perp})/\{\pm \id\}$;}
    \item{if $\omega =\left(\frac{2(m-1)}{\gamma}, \gamma\right)=1$ with $\gamma \in\{1, 2\}$ and $m>2$, and we suppose $\widehat O(\Lambda_{K3^{[m]}}, h)\norm O(h^\perp)$, then $G\simeq O(A_{h^\perp})/ \left\langle  s, -\id \right\rangle$. For instance, this is the case for $(m-1,d)=1$.}
\end{itemize}

In Theorem~\ref{thm:divisori_ref}, we showed that the ramification divisors of $\cover$ are parametrized by vectors $\beta\in h^\perp$ that define nontrivial reflections in $G$.
As in the case of polarized K3 surfaces studied in \cite{stellari2008}, we would like to  characterize these vectors $\beta$, at least in some cases.\\

For $\gamma>2$, the following result allows us to characterize the ramification divisor of $\cover$.
\begin{proposition}[{\cite[Proposition~3.1]{GritsenkoRefl2007}}]
Let $\beta\in \Lambda$ be a primitive vector that defines a reflection. The isometry $r_\beta$ acts as $\id$ on the discriminant group $\beta^\perp$ if and only if $\beta^2 = -2$.
\end{proposition}

We now assume $\gamma = 1$. Our first result, Theorem~\ref{thm:HK_beta_reflection_numerical}, characterizes vectors $\beta$ that define trivial reflections. Corollary~\ref{cor:div_irred_HK1} then gives a list of all divisorial components of the $G$-Galois cover $\cover$ of \eqref{eqn:cover_HK}. \\

As in Section~\ref{sec:isometries_lattices_nomality}, we consider a lattice $L_{2t} =M \oplus U \oplus \Z \ell$ and a vector $h\in L_{2t}$ of square $2d$ and divisibility $\gamma=1$.  In this case the lattice  $\Lambda = h^\perp$ is isomorphic to
$$ M \oplus \Z k \oplus  \Z \ell,$$
where $k =  e - df$, and with $k^2 = -2d$ and $\ell^2=-2t$ (see Remark~\ref{rmk:gamma1}). Proposition~\ref{prop:dec_disc_group} shows that the discriminant group $A_\Lambda$ is isomorphic to
$$
    A_\Lambda \simeq \langle \bar k_1\rangle \times\langle \bar k_2\rangle\simeq \Z/2d\Z \times \Z/2t\Z,
$$
where $k_1 = \frac{e+df}{2d}-f = \frac{e-df}{2d}=k_*$ and $k_2 = \frac{\ell}{2t}=\ell_*$.\\

Each primitive vector $\beta\in\Lambda$ can be written as
$$
    \beta = \cm\vm + \ck k + \cl \ell
$$
where $\cm, \ck, \cl$ are relatively prime integers and $\vm\in M$ is a primitive vector. Such a vector $\beta$ has divisibility
$\div(\beta) = (\cm , 2d\ck, 2t\cl)$ and square
\begin{equation}\label{eqn:HK_square}
\beta^2 = \cm^2\vm^2 -2d\ck^2 -2t\cl^2.
\end{equation}
The vector $\beta$ defines a reflection if and only if $\beta^2\mid 2\div(\beta)$. Observe that this implies
\begin{equation}\label{eqn:beta2divides}
\beta^2\mid 4d\ck \qquad \text{ and } \qquad \beta^2\mid 4t\cl.
\end{equation}
Since $\beta\cdot k_* = -\ck$, we obtain
\begin{align*}
    [r_\beta(k_*)] &= \left[ k_* - 2\frac{\beta\cdot k_*}{\beta^2}\beta\right]\\
                &= \left[k_* + 2\frac{\ck}{\beta^2}(2d\ck k_* + 2t\cl \ell_*)\right]\\
                &=\left[\left(1 + \frac{4d\ck^2}{\beta^2}\right)k_* + \frac{4t\cl\ck}{\beta^2} \ell_* \right]
\end{align*}
in $A_\Lambda$, and an analogous computation gives $[r_\beta(\ell_*)]$. Hence $[r_\beta]\in  O(A_\Lambda)$ is the matrix

\begin{equation}\label{eqn:matrix_HK}
\begin{pmatrix}
\left[1 + \dfrac{4d\ck^2}{\beta^2}\right]_{2d} & \left[\dfrac{4d\ck\cl}{\beta^2}\right]_{2d}\\[10pt]
\left[\dfrac{4t\cl\ck}{\beta^2}\right]_{2t} &\left[1 + \dfrac{4t\cl^2}{\beta^2}\right]_{2t}
\end{pmatrix} \in  O(\Z/2d\Z \times \Z/2t\Z),
\end{equation}
where the entries of the matrix \eqref{eqn:matrix_HK} are integers by \eqref{eqn:beta2divides}.

If $\widehat O(L_{2t}, h)\norm O(h^\perp)$, the group $G$ is isomorphic to $O(A_\Lambda)/\{\pm s, \pm \id\}$, where $s = \begin{pmatrix} 1& 0 \\ 0 & -1\end{pmatrix}$.\\
The next theorem characterize vectors $\beta$ that define a reflection $r_\beta$ trivial in $G$.
\begin{theorem}\label{thm:HK_beta_reflection_numerical}
    Let $\beta\in M \oplus \Z(-2d) \oplus \Z(-2t)$ be a primitive vector with $\beta^2<0$. Let $k$ be a generator of the factor $\Z(-2d)$ and let $\ell$ be a generator of the factor $\Z(-2t)$.\\
    The vector $\beta$ defines a reflection $r_\beta$ such that $[r_\beta]$ is contained in the group $\{\pm s, \pm \id\}$ if and only if $\beta$ satisfies both conditions:
    \begin{enumerate}[label=\rm{\alph*})]
        \item\label{en:thm:relazioni:cond1} {$\beta^2\mid 2\mathrm{div}(\beta)$;}
        \item \label{en:thm:relazioni:cond2}{one has
                        \begin{itemize}[label= -]
                            \item either $\beta^2 = -2$;
                            \item or $\beta^2 = -2t$ and $2td\mid \beta\cdot k$;
                            \item or $\beta^2 = -2d$ and $2td\mid \beta\cdot \ell$;
                            \item or $\beta^2 = -2td$, $(t,d)=1$, and $2td\mid (\beta\cdot k, \beta\cdot \ell)$.
                        \end{itemize}}
    \end{enumerate}
\end{theorem}
\begin{proof}
    The vector $\beta$ defines a reflection if and only if $\beta^2\mid 2\div(\beta)$. We have
    \begin{equation}\label{eqn:inv_divisibilityDividesLcm}
        \div(\beta) = (\cm, 2d\ck, 2t\cl) \mid 2\lcm(t,d)(a,b,c) = 2\lcm(t,d),
    \end{equation} where the last equality holds because $\beta$ is primitive. Therefore,
    \begin{equation}\label{eqn:beta^2divideslcm}
    \beta^2\mid 4\lcm(t,d).
    \end{equation}

    We want to characterize those vectors $\beta$ such that
    \begin{equation}\label{eqn:refelction_trivial_HK_gamma1}
    [r_\beta] \in \left\{ \pm\begin{pmatrix} 1& 0 \\ 0 & -1\end{pmatrix}, \pm \id\right\},
    \end{equation}
    where the matrix $[r_\beta]$ is given in equation~\eqref{eqn:matrix_HK}.\\

    Assume \eqref{eqn:refelction_trivial_HK_gamma1} holds.  The off-diagonal terms are then zero, namely $2t\beta^2\mid 4t\ck\cl$ and $2d\beta^2\mid 4d\ck\cl$, or equivalently
    \begin{equation}\label{eqn:HK_off_diag}
         \beta^2\mid 2\ck\cl.
    \end{equation}

    As for the diagonal terms, we want to understand when they are equal to $\pm 1$. For the first entry, we have
    \begin{enumerate}[label = (\Alph*)]\label{enum:HK_diag}
        \item \label{enum:HK_diag_1} $1 + \dfrac{4d\ck^2}{\beta^2}\equiv  1 \pmod{2d}$ if and only if $2d\beta^2\mid 4d\ck^2$, or equivalently $ \beta^2\mid 2\ck^2$.
         \item \label{enum:HK_diag_-1} $1 + \dfrac{4d\ck^2}{\beta^2}\equiv  -1 \pmod{2d}$ if and only if $2d\beta^2\mid 2(\beta^2 +2d\ck^2) $, so exactly when
         \begin{equation}\label{eqn:HK_diag_-1}
            d\beta^2\mid \beta^2 +2d\ck^2,
         \end{equation}
         which yields $d\mid \beta^2$ and $\beta^2\mid 2d\ck^2$. We show that necessarily $2d\mid \beta^2$. Indeed, if not, $\beta^2$ and $d$ have the same valuation at $2$, hence $d$ is even and, from $\beta^2\mid 2d\ck^2$,
         we obtain $\beta^2\mid d\ck^2$. Therefore, by \eqref{eqn:HK_diag_-1}, the even number $d$ divides the odd number $1 + 2\frac{d\ck^2}{\beta^2}$, and clearly it is not possible.\\
         In conclusion, $1 + \dfrac{4d\ck^2}{\beta^2}\equiv  -1 \pmod{2d}$ implies $2d\mid \beta^2$ and $\beta^2\mid 2d\ck^2$.
    \end{enumerate}
    The same argument applied to the second diagonal term yields analogous results with $t$ and $\cl$ in place of $d$ and $\ck$ respectively. Namely, we have
    \begin{enumerate}[label = (\Alph*')]\label{enum:HK_diag}
        \item \label{enum:HK_diag_1'} $1 + \dfrac{4t\cl^2}{\beta^2}\equiv  1 \pmod{2t}$ if and only if $\beta^2\mid 2\cl^2$.
         \item \label{enum:HK_diag_-1'} $1 + \dfrac{4d\ck^2}{\beta^2}\equiv  -1 \pmod{2d}$ if and only if
         \begin{equation}\label{eqn:HK_diag_-1'}
            t\beta^2\mid \beta^2 +2t\cl^2.
         \end{equation}
         Moreover, the last condition implies $2t\mid \beta^2$ and $\beta^2\mid 2t\cl^2$.
    \end{enumerate}

    Putting all together, we can characterize primitive vectors $\beta$ that satisfies a) and such that $[r_\beta]\in \{\pm \id, \pm s\}$.
    \begin{itemize}
        \item $[r_\beta] = [\id]$ if and only if $\beta^2 = -2$.\\ Clearly, for $\beta^2=-2$, the matrix \eqref{eqn:matrix_HK} is the identity. Conversely, if $[r_\beta] = [\id]$, the two diagonal terms are equal to 1, and \ref{enum:HK_diag_1} and \ref{enum:HK_diag_1'} give
        $$
        \beta^2\mid 2\ck^2 \qquad\text{and}\qquad \beta^2\mid 2\cl^2.
        $$
        Moreover $\beta^2\mid 2\div(\beta)\mid 2\cm$. Since $(\cm,\ck,\cl)=1$, we obtain $\beta^2\mid 2$ hence the only possibility is $\beta^2=-2$, as we assumed $\beta^2<0$.
        \item $[r_\beta] = s$ if and only if $\beta^2 = -2t$ and $t\mid \beta\cdot (k/2d)$.\\
        Indeed, if $[r_\beta] = s$, the second diagonal term must be equal to $-1$, hence from \ref{enum:HK_diag_-1'} we obtain
        $$
        \beta^2 = 2t i \text{ for some negative integer $i$, with}  \quad 2ti=\beta^2\mid 2t\cl^2,
        $$
        hence $i\mid \cl^2$. Moreover, since the first diagonal term is 1, \ref{enum:HK_diag_1} implies $2ti = \beta^2\mid 2\ck^2$, and therefore $i\mid \ck^2$. Finally, we also have  $2ti=\beta^2\mid 2\div(\beta)\mid 2\cm$, hence
        $$
        i \mid (\cm, \ck^2, \cl^2)=1.
        $$
        Thus we obtain $\beta^2 = -2t$. The number $\beta^2$ divides both $2\ck^2$ and $2\cm$. Since, by \eqref{eqn:HK_off_diag}, $\beta^2$ also divides $2\ck\cl$, we obtain $-2t=\beta^2\mid 2(\cm, \ck^2, \ck\cl)\mid 2\ck$, which implies
        $$t\mid -\ck = \beta\cdot (k/2d).$$
        Conversely, for $\beta^2=-2t$ with $\beta^2\mid 2\div(\beta)$ and $t\mid \ck = \beta\cdot (k/2d)$, we show that $[r_\beta]$ is equal to $s$. Indeed, the only nontrivial check is to show that  $\left[1 + \dfrac{4t\cl^2}{\beta^2}\right]_{2t}$ is equal to $[-1]_{2t}$. By equation~\eqref{eqn:HK_diag_-1'}, this condition can be rewritten as $t\beta^2\mid \beta^2 + 2t\cl^2$. By equation~\eqref{eqn:HK_square}, we have
        $$
             \beta^2 + 2t\cl^2 = \cm^2\vm^2 -2d\ck^2,
        $$
        where $2t=-\beta^2\mid 2\div(\beta)\mid 2\cm$, $t\mid \ck$ by hypothesis, and $m^2$ is even. Hence, we obtain $t\beta^2=-2t^2\mid \cm^2\vm^2 -2d\ck^2 = \beta^2 + 2t\cl^2$, which is the condition we needed.
        \item Analogously, $[r_\beta] = -s$ if and only if $\beta^2 = -2d$ and $d\mid \beta\cdot (\ell/2t)$.
        \item $[r_\beta] = -\id$ if and only if $\beta^2 = -2td$ and $\beta^2\mid (\beta\cdot k, \beta\cdot l)$.\\

        If $[r_\beta] = -\id$, the diagonal terms must be equal to $-1$. From \ref{enum:HK_diag_-1} and \ref{enum:HK_diag_-1'} we obtain
        $$2d\mid \beta^2 \qquad\text{and}\qquad \beta^2\mid 2d\ck^2,$$
        $$2t\mid \beta^2 \qquad\text{and}\qquad \beta^2\mid 2t\cl^2.$$
        Therefore, we get $2\lcm(t,d)\mid \beta^2$, which in turns divides $4\lcm(t,d)$ by \eqref{eqn:beta^2divideslcm}, so $\beta^2$ is either $-2\lcm(t,d)$ or $-4\lcm(t,d)$.\\
        We exclude the case $\beta^2 = -4\lcm(t,d)$. Indeed, in this case, from $\beta^2\mid 2d\ck^2$, we obtain $\frac{4\lcm(t,d)}{2d}\mid \ck^2$, hence $2\mid \ck^2$, and analogously, from $\beta^2\mid 2t\cl^2$, we get $2\mid \cl^2$. Since we also have $2\mid \cm$, because $\beta^2\mid 2\div(\beta)\mid 2\cm$, we get a contradiction, $\beta$ being primitive.\\

        Therefore $\beta^2 = -2\lcm(t,d)$. If we denote by $z$ the number $(t,d)$, and we write $t = z\tau$ and $d=z\delta$, then $(\tau, \delta)=1$ and $\lcm(t,d)=z\tau\delta$. We show that, still under the hypothesis $[r_\beta]=-\id$, we have $z=1$.\\
        Condition \eqref{eqn:HK_diag_-1} is equivalent to $1+\dfrac{2d\ck^2}{\beta^2}\equiv 0\pmod{d}$. Thus,
        $$
            1 - \frac{2z\delta\ck^2}{2z\tau\delta} \equiv 0 \pmod{z}, \qquad\text{hence}\qquad \frac{\ck^2}\tau\equiv 1 \pmod{z}.
        $$
        In particular, we can write $\ck^2 = \tau\ck_1$, where $(\ck_1,z)=1$. Analogously, using \eqref{eqn:HK_diag_-1'}, we show that there exists $\cl_1$ coprime with $z$ such that $\cl^2 = \delta\cl_1$. The vanishing of the off-diagonal terms condition (see equation \eqref{eqn:HK_off_diag}) gives
        $$
        2z\tau\delta = -\beta^2\mid 2\ck\cl.
        $$
        Hence, for each prime $q$ that divides $z$, we have
        \begin{equation}\label{eqn:HK_conto_valutazioni}
            v_q(z) + v_q(\tau\delta) = v_q(z\tau\delta) \le v_q(\ck\cl) = \frac{v_q(\ck^2\cl^2)}2= \frac{v_q(\tau\ck_1\delta \cl_1)}2 = \frac{v_q(\tau\delta)}{2},
        \end{equation}
        where in the last equality we used that $v_q(\ck_1)=v_q(\cl_1)=0$ because $q\mid z$ and $z$ is coprime to both $\ck_1$ and $\cl_1$. Equation~\eqref{eqn:HK_conto_valutazioni} implies $v_q(z)\le 0$, which is absurd since $q\mid z$. Hence $z=1$.

        Therefore, we have $(t,d)=1$ and $\beta^2 =-2td$. The divisibility relations $\beta^2\mid 2\div(\beta)\mid 4d\ck$ and $\beta^2\mid 2d\ck^2$ imply
        $$
            -2td = \beta^2\mid (4d\ck, 2d\ck^2) = 2d(2\ck, \ck^2),
        $$
        and thus $t\mid (2\ck, \ck^2)$. Moreover, from equation \eqref{eqn:HK_off_diag}, we have $-2td=\beta^2\mid 2\ck\cl$, therefore we obtain
        $$
        t\mid (\ck\cl, 2\ck, \ck^2) = \ck(\cl, 2, \ck).
        $$
        We prove that this implies $t\mid \ck$. If $(\cl, 2, \ck)=1$, the statement is clear. Otherwise, since $2\mid (\ck, \cl)$ and $\beta$ is primitive, then $2\nmid \cm$. Therefore, since $-2td =\beta^2\mid2\div(\beta)\mid 2\cm$, we obtain that $t$ is odd and hence if $t$ divides $2\ck$, then it also divides $\ck$.\\
        Analogously, we obtain  $d\mid \cl$. Since $\beta\cdot k = -2d\ck$ and $\beta \cdot \ell = -2t\cl$, we obtain the required condition $\beta^2\mid (\beta\cdot k, \beta \cdot \ell)$.\\

        Conversely, as in the case $[r_\beta]=s$, direct computations show that if $\beta$ is a primitive vector that defines a reflection of square $\beta^2=-2td$ and such that $\beta^2\mid (\beta\cdot k, \beta\cdot \ell)$, then $[r_\beta]=-\id$. First observe that the above conditions imply $(t,d)=1$.
        Indeed, we have $\beta\cdot k = -2d\ck$ and $\beta\cdot \ell= -2d\cl$. Hence $\beta^2\mid (\beta\cdot k, \beta\cdot \ell)$ is equivalent to $2td\mid 2d\ck$ and $2td\mid 2t\cl$, and namely to $t\mid \ck$ and $d\mid \cl$. Since $\beta^2\mid 2\div(\beta)\mid 2\cm$, it follows that $(t,d)$ divides $\cm$. It also divides  $\ck$ and $\cl$, from the previous observation. Since $\beta$ is primitive, this implies $(t,d)=1$.\\ Finally, equation \eqref{eqn:HK_off_diag} is easily verified and the computation for  equations \eqref{eqn:HK_diag_-1} and \eqref{eqn:HK_diag_-1'} is the same as in the case $[r_\beta]=s$.
    \end{itemize}
\end{proof}

We observe that the conditions found on $\beta$ are invariant under the action of $\widehat O(\Lambda_{K3^{[m]}}, h)$. Indeed, if $g\in \widehat O(\Lambda_{K3^{[m]}}, h)$, we have $g(l) = \pm l + 2t n$ and $g(k)= k + 2d n'$ for some $n, n'\in \Lambda$. 

Therefore, if $\beta^2 = -2d$, then $2td\mid \beta\cdot l$ if and only if $2td\mid g(\beta)\cdot l$. Indeed,
$$
    \beta\cdot l = g(\beta)\cdot g(l)= g(\beta)\cdot(\pm l) + 2t g(\beta)\cdot n,
$$
and, since $-2d = \beta^2\mid 2\div(\beta)=2\div(g(\beta))$, we have $2td\mid 2t g(\beta)\cdot n$. The invariance of the other conditions can be shown in a similar way using $g(k)= k + 2d n'$.\\

The next corollary is our main result: we determine the ramification divisors of the Galois cover \eqref{eqn:cover_HK}, in the case of polarized hyper-Kähler manifolds of polarization type of square $2d$ and divisibility $1$. It applies in particular when $(m-1,d)=1$ (see Remark~\ref{rmk:HK_summary_normality_subgroup}).
\begin{corollary}\label{cor:div_irred_HK1}
Let $h\in \Lambda_{K3^{[m]}}$ be a primitive vector of square $2d$ and divisibility $1$ such that $\widehat O(\Lambda_{K3^{[m]}}, h)$ is a normal subgroup of $O(h^\perp)$.
The divisorial components of the ramification locus of $\cover: \cD_{h^\perp}/{\widehat O(\Lambda_{K3^{[m]}}, h^\perp)} \map \cD_{h^\perp}/O(h^\perp)$ are the Heegner divisors $\Hh_{\beta^\perp}$ such that $\beta$ is primitive and satisfies both conditions
\begin{enumerate}[label=\rm{\alph*})]
    \item {$\beta^2\mid 2\div(\beta)$;}
    \item {$\beta^2$ is such that:
                        \begin{itemize}
                            \item $\beta^2 \neq -2$;
                            \item if $\beta^2 = -2(m-1)$, then $2(m-1)d\nmid \beta\cdot k$;
                            \item if $\beta^2 = -2d$, then $2(m-1)d\nmid \beta\cdot \ell$;
                            \item if $\beta^2 = -2(m-1)d$, then $2(m-1)d\nmid (\beta\cdot k, \beta\cdot \ell)$.
                        \end{itemize}}
\end{enumerate}
\end{corollary}

\section{Hyper-Kähler fourfolds}
We now restrict to the case $m=2$ of hyper-Kähler fourfolds of polarization type $\tau = O(\Lambda_{K3^{[2]}})h$, where $h$ is a primitive vector of square $2d$ and divisibility $\gamma$. Since $\gamma\mid 2(m-1)$, we obtain that $\gamma$ is either $1$ or $2$.
In this case the group $\widehat O(\Lambda_{K3^{[2]}})$ is a normal subgroup of $O(h^\perp)$, and defines the Galois cover
$$
\cover: \cD_{h^\perp}/\widehat O(\Lambda_{K3^{[2]}}, h^\perp) \map \cD_{h^\perp}/ O(h^\perp),
$$
with Galois group $G = O(A_{h^\perp})/\{\pm\id\}$.\\

Corollary~\ref{cor:div_irred_HK1} implies that, when $\gamma = 1$, the divisorial components of the ramification locus of $\cover$
are the Heegner divisors $\Hh_{\beta^\perp}$ such that $\beta$ is primitive and satisfies the conditions
\begin{samepage}
\begin{enumerate}[label=\rm{\alph*})]
    \item {$\beta^2\mid 2\div(\beta)$;}
    \item {$\beta^2\neq -2$ and if $\beta^2 = -2d$, then $2d\nmid \beta\cdot \ell$.}
\end{enumerate}
\end{samepage}

Observe that, from Equation~\eqref{eqn:beta^2divideslcm}, if $\beta$ defines a reflection, then $\beta^2\mid 4d$.\\

In \cite{Image_per_map_HK4}, Debarre and Macrì characterized the image of the period morphism of polarized hyper-Kähler fourfolds of square $2d$ and divisibility $\gamma$.  We would like to characterize those ramification divisors that meet its image: since in this case the period morphism is an embedding (see \cite[Proposition~3.2]{Song}), they induces a nonzero divisor on the moduli space.\\

For each primitive rank-2 sublattice $K$ of $\Lambda_{K3^{[2]}}$ of signature $(1,1)$ that contains the vector $h$, the authors denote by $\cD_{2d, K}^{(1)}$ the divisor of $\cD_{h^\perp}/\widetilde O(h^\perp)$ cut out by the codimension-2 subspace $\Proj(K^\perp_\C) \subset \Proj((\Lambda_{K3^{[2]}})_\C)$. Namely, if $K\cap h^\perp = \Z \beta$ for some primitive vector $\beta\in h^\perp$, the divisor $\cD_{2d, K}^{(1)}$ is the Heegner divisor $\Hh_{\beta^\perp}$.
Moreover, for each positive integer $D$, the authors set
$$
    \cD_{2d, D}^{(1)} \coloneqq \bigcup_{\mathrm{disc}(K^\perp) = D} \cD_{2d, K}^{(1)} \subset \cD_{h^\perp}/\widetilde O(h^\perp).
$$

The image of the period morphism
$$
^{[2]}\wp: {^{[2]}M_{2d}^{(1)}} \longinj \cD_{h^\perp}/\widetilde O(h^\perp)
$$
of polarized hyper-Kähler fourfolds of K$3^{[2]}$-type and polarization type defined by a vector $h$ of square $2d$ and divisibility $1$ is described in \cite[Theorem~6.1]{Image_per_map_HK4}. In particular, they show that the following holds.

\begin{proposition}[{\cite[Theorem~6.1]{Image_per_map_HK4}}]\label{prop:DM_image_per_morph}
The image of the period morphism $\wp_{K3^{[2]}}$ is the complement of certain irreducible Heegner divisors contained in the hypersurfaces $\cD_{2d, 2d}^{(1)}$, $\cD_{2d, 8d}^{(1)}$, $\cD_{2d, 10d}^{(1)}$ and $\cD_{2d, \frac{2d}{5}}^{(1)}$, where the last case occurs only for $d \equiv \pm 5 \pmod{25}$.
\end{proposition}

We now determine when a Heegner divisor $\Hh_{\beta^\perp}$ is contained in one of these hypersurfaces, for $\beta \in h^\perp$ primitive vector of negative square that defines a reflection.

\begin{proposition}\label{prop:inv_div_contained_image}
Let $\beta$ be a primitive vector that defines a reflection and such that
\begin{itemize}
    \item $\beta^2 \neq -2, \beta^2 \neq -8$;
    \item and if $d \equiv \pm 5 \pmod{25}$, $\beta^2\neq -10$ and $\beta^2 \neq -45$.
\end{itemize}
Then the Heegner divisor $\Hh_{\beta^\perp}$ induces a nonzero divisor in the moduli space $^{[2]}M_{2d}^{(1)}$.
\end{proposition}
\begin{proof}
Observe that, if $K\cap h^\perp = \Z \beta$ for some vector $\beta$ of negative square, the lattices $K^\perp$ and $\langle h, \beta \rangle^\perp$ are equal. In particular, using \cite[Lemma~7.2]{GHS13}, we can compute

\begin{equation}\label{eq:HK2_discriminant_betaperp}
   \disc(K^\perp)= \disc(\langle h, \beta \rangle^\perp) = \frac{-\beta^2 \disc{(h^\perp)}}{\div(\beta)^2} = \frac{-4d\beta^2}{\div(\beta)^2},
\end{equation}
where we used that $\disc{(h^\perp)} = |A_{h^\perp}| = 2d\cdot 2$.\\
Therefore, the Heegner divisor $\Hh_{\beta^\perp}$ is contained in the locus $\cD_{2d, \frac{-4d\beta^2}{\div(\beta)^2}}^{(1)}$.\\

If $\beta$ is a primitive vector of negative square that defines a reflection, then $\beta^2\mid 2\div(\beta)$ and, since $\div(\beta)$ always divides $\beta^2$, the integer $\beta^2$ is equal to either $-\div(\beta)$ or $-2\div(\beta)$. Hence,
\begin{itemize}
    \item when $\beta^2 = -\div(\beta)$, formula \eqref{eq:HK2_discriminant_betaperp} yields
    $$
        \disc(\langle h, \beta \rangle^\perp) = -\frac{4d}{\beta^2},
    $$
    where $-\beta^2 = \div(\beta)\mid 2d$.
Hence the Heegner divisor $\Hh_{\beta^\perp}$ is contained in the locus $\cD_{2d, -2\frac{2d}{\beta^2}}^{(1)}$.
Proposition~\ref{prop:DM_image_per_morph} implies that, if
$$
    -2\frac{2d}{\beta^2}\not \in \bigg\{2d, 8d, 10d, \frac{2d}{5}\bigg\},
$$
where the last case only occurs for $d\equiv \pm \pmod{25}$, the Heegner divisor $\Hh_{\beta^\perp}$ meets the image of $^{[2]}\wp$.
Namely,
\begin{center} if $\beta^2\neq -2$ and, for $d\equiv \pm 5 \pmod{25}$, $\beta^2\neq -10$,
\end{center} the divisor $\Hh_{\beta^\perp}$ defines a nonzero divisor of the moduli space $^{[2]}M_{2d}^{(1)}$.

\item when $\beta^2 = -2\div(\beta)$, formula \eqref{eq:HK2_discriminant_betaperp} yields
    $$
        \disc(\langle h, \beta \rangle^\perp) = -\frac{16d}{\beta^2},
    $$
    where $-\beta^2 = 2\div(\beta)\mid 4d$.
    Hence the Heegner divisor $\Hh_{\beta^\perp}$ is contained in the locus $\cD_{2d, -2\frac{8d}{\beta^2}}^{(1)}$.
    Proposition~\ref{prop:DM_image_per_morph} implies that, if
    $$
        -2\frac{8d}{\beta^2}\not \in \bigg\{2d, 8d, 10d, \frac{2d}{5}\bigg\},
    $$
    where the last case only occurs for $d\equiv \pm 5\pmod{25}$, the Heegner divisor $\Hh_{\beta^\perp}$ meets the image of $^{[2]}\wp$.
    Namely,
    \begin{center}
    if $\beta^2\neq -2$, $\beta^2\neq -8$, and, for $d\equiv \pm 5\pmod{25}$, if $\beta^2\neq -45$,
    \end{center}the divisor $\Hh_{\beta^\perp}$ defines a nonzero divisor of the moduli space $^{[2]}M_{2d}^{(1)}$.
\end{itemize}
\end{proof}

\subsection{Hyper-Kähler fourfolds with polarization of square $2$}
We consider the polarization type defined by a vector $h$ of square $2$ ($d=1$). In this case, $\gamma$ is $1$ (see Remark~\ref{rmk:gamma_t_d}), the group of isometries of $A_{h^\perp}$ is
$$
O(A_{h^\perp}) = \bigg\langle\begin{pmatrix} 0 & 1 \\ 1 & 0\end{pmatrix}\bigg\rangle \simeq \Z/2\Z,
$$
and, since $-\id$ and $\id$ define the same isometry of $A_{h^\perp}$, we have $G \simeq Z/2\Z$.

If we write $h^\perp= M \oplus \Z k \oplus \Z \ell$, where $k^2 = -2$ and $\ell^2 = -2$, the vector $\gamma = k + \ell$ is a vector of square $-4$ that defines a nontrivial reflection in $G$, hence $G= [r_\gamma]$.

\begin{corollary}
Let $h\in \Lambda_{K3^{[2]}}$ be a primitive vector of square $2$. The ramification divisor of the double cover
$$
\cover:  {^{[2]}\Pp_{2}^{(1)}}=\cD_{h^\perp}/\widehat O(\Lambda_{K3^{[2]}}, h^\perp) \map \cD_{h^\perp}/ O(h^\perp)
$$
is irreducible and meets the image of the
period morphism
$$
    ^{[2]}\wp_{2}^{(1)}: {^{[2]}\Mm_{2}^{(1)}} \longinj {^{[2]}\Pp_{2}^{(1)}}.
$$
\end{corollary}
\begin{proof}
The components of the ramification divisor of $\cover$ are the irreducible divisors $D$ of $\cD_{h^\perp}/\widetilde O(h^\perp)$ contained in the fixed locus of some nontrivial element of $G$.\\
The only nontrivial element of $G$ is $[r_\gamma]$, where $\gamma = k + l$. Theorem~\ref{thm:divisori_ref} shows that if $D$ is an irreducible divisor contained in $\mathrm{Fix}([r_\gamma])$, there exists a vector $\beta$ that defines a nontrivial reflection in $G$ such that $D = \Hh_{\beta^\perp}$ and $[r_\gamma]=[r_\beta]$. Now, since $[r_\beta]$ is nontrivial, we have $\beta^2\neq -2$ and hence $\beta^2 = -4$ (see \eqref{eqn:beta^2divideslcm}). We show that all primitive vectors $\beta\in \Lambda$ of square $-4$ that define a reflection are conjugate by an element of $\widetilde O(\Lambda)$. In particular this implies $\Hh_{\beta^\perp} = \Hh_{\gamma^\perp}$ and that the ramification divisor of $\cover$ is therefore irreducible.

From Eichler's Lemma~\ref{lemma:eichler}, we know that the $\widetilde O(\Lambda)$-orbit of a vector $\beta$ is uniquely determined by $\beta^2$ and $\beta_*\in A_{h^\perp}$. Notice that, for each primitive vector $\beta$ of square $-4$ that defines a reflection, $\div(\beta)=2$. Indeed, Equation~\eqref{eqn:inv_divisibilityDividesLcm} implies $\div(\beta)\mid 2$, and $-4 =\beta^2\mid 2 \div(\beta)$ because $\beta$ defines a reflection.\\
We write $\beta = \cm\vm +\ck k + \cl l$, where $\cm, \ck, \cl$ are integers and $\vm\in M$ is a primitive vector. Recall that $A_{h^\perp} = \langle k_*  \rangle \times \langle l_* \rangle$ where $k_* = \left[\frac{k}2\right]$ and $l_* = \left[\frac{l}2\right]$. Therefore,
$$
    \beta_* = \left[\frac{\beta}{\div(\beta)}\right] = \bar{\ck} k_* + \bar{\cl} l_* \in A_{h^\perp}.
$$
We show that $\bar{\ck}=\bar{\cl}=1\in \Z/2\Z$, hence $\beta_* = k_* + l_*$. This is enough to finish the proof.

Since $\div(\beta) = (\cm, 2\ck, 2\cl) = 2$, we can write $\cm= 2\cm_1$ for some integer $a_1$. By computing the square of $\beta$
$$
     -4 = \beta^2 = (2\cm_1)^2 \vm^2 - 2\ck^2 - 2 \cl^2
$$
we obtain $2\mid \ck^2+\cl^2$, from which we obtain that $\ck$ and $\cl$ have the same parity. Since $\cm$ is even and $\cm$, $\ck$, $\cl$ are coprime, it follows that $\ck$ and $\cl$ are both odd.

Finally, Proposition~\ref{prop:inv_div_contained_image} shows that $\Hh_{\gamma^\perp}$ meets in the image of $^{[2]}\wp_{2}^{(1)}$.
\end{proof}

The moduli space $^{[2]}M_{2}^{(1)}$ contains a dense open subset $U_{2,1}$ which is the moduli space of double EPW sextics (see \cite[Example~3.5]{Deb-HK}). The involution $[r_\gamma] \in G$ defines an involution on $U_{2,1}$ which is the duality involution of double EPW sextics studied by O'Grady in \cite[Theorem~1.1]{OGrady08}. Observe that the associated ramification divisor $\Hh_{\gamma^\perp} = \cD_4$ does not meet the image of $U_{2,1}$ (see \cite{O'GPeriodsEPW}).\\

As observed in \cite[Section 3.9]{Deb-HK}, the quotient $\cD_{h^\perp}/O(h^\perp)$ is indeed the period space ${^{[3]}\Pp_{4}^{(2)}}$ of polarized hyper-Kähler manifolds of K$3^{[3]}$-type with polarization of square 4 and divisibility $2$. Indeed, given a primitive vector $h_2\in \Lambda_{K3^{[3]}}$ of square $4$ and divisibility $2$, the lattice $h_2^\perp$ is isomorphic to $h^\perp$, and direct computations show $\widehat O(\Lambda_{K3^{[3]}}, h_2) = O(h_2^\perp)$. \cite{KKM-EPW} shows that the cover $\cover$ associates to a double EPW sextic the corresponding double EPW cube.

 \bibliography{Bibliografia}
 \bibliographystyle{alpha}

\end{document}